\documentclass{amsart}
\usepackage{latexsym,amssymb,amsthm,amsmath,amscd}

\theoremstyle{plain}
\newtheorem{theorem}{Theorem}[section]

\newtheorem{proposition}{Proposition}[section]

\newtheorem{corollary}{Corollary}[section]

\newtheorem{example}{Example}[section]
\numberwithin{equation}{section}

\theoremstyle{remark}
\newtheorem{remark}{Remark}[section]

 \numberwithin{equation}{section}

\title[Warped product submanifolds]
{A survey on geometry of warped product submanifolds}

\author{ Bang-Yen Chen}
\address{\it Michigan State University \newline\indent
Department of Mathematics  \newline\indent
619 Red Cedar Road,  \newline\indent East Lansing, Michigan
48824--1027, U.S.A.}
\email{bychen@math.msu.edu}


\begin{document}

\begin{abstract} 
The warped product $N_1\times_f N_2$ of two Riemannian manifolds $(N_1,g_1)$ and $(N_2,g_2)$ is the product manifold $N_1\times N_2$ equipped with the warped product metric
$g=g_1+f^2 g_2$, where $f$ is a positive function on $N_1$. 
The notion of warped product manifolds is one of the most fruitful generalizations of Riemannian products. Such notion plays very important roles in differential geometry as well as in
physics, especially in general relativity. 
Warped product manifolds have been studied for a long period of time. In contrast, the study of warped product submanifolds was only initiated around the beginning of this century   in a series of articles \cite{c5.1,c7,c8,c02-2}.  Since then the study of  warped product submanifolds has become a very active research subject.

 In this article we survey important results on warped product submanifolds in various ambient manifolds.   It is the author's hope that this survey article will provide a good introduction on the theory of warped product submanifolds as well as a useful reference for further research on this vibrant research subject.

\end{abstract}

\maketitle
\pagestyle{myheadings}
\markboth{\hspace{0,45 cm}\hrulefill\ B.-Y. Chen}{{\it Warped product submanifolds} \hrulefill\hspace{0,45 cm} }

\setcounter{equation}{0}

{\bf Table of Contents}

\begin{enumerate}
\item[1.]  Introduction \dotfill 1                            

\item[2.] Preliminaries\dotfill 2
\item[3.]  Warped product submanifolds of Riemannian manifolds \dotfill 6
\item[4.]  Multiply warped product submanifolds \dotfill 8
\item[5.]  Arbitrary warped products in complex space forms \dotfill 10
\item[6.]  Warped products as K\"ahlerian submanifolds \dotfill 13
\item[7.] $CR$-products in K\"ahler manifolds \dotfill 14
\item[8.] Warped product Lagrangian submanifolds of K\"ahler manifolds \dotfill 15
\item[9.]  Warped Product $CR$-submanifolds of K\"ahler manifolds \dotfill  16
\item[10.] $CR$-warped products with compact holomorphic factor \dotfill  18
\item[11.]  Another optimal inequality for $CR$-warped products \dotfill 20
\item[12.] Warped Product $CR$-submanifolds and $\delta$-invariants \dotfill  21
\item[13.]  Warped product real hypersurfaces in complex space forms \dotfill  22
\item[14.]  Warped Product $CR$-submanifolds of nearly K\"ahler manifolds \dotfill  25
\item[15.] Warped product submanifolds in para-K\"ahler manifolds \dotfill  26
\item[16.] Contact $CR$-warped product submanifolds in Sasakian manifolds. \dotfill  28
\item[17.] Warped product submanifolds in affine spaces \dotfill  30
\item[18.] Twisted product submanifolds \dotfill  35

\item[19.] Related articles \dotfill  37

\item[20.] References \dotfill 38
\end{enumerate}

\section{\uppercase{Introduction}}

\def\<{\left < }
\def\>{\right >}

Let $B$ and $F$ be two Riemannian manifolds equipped with Riemannian metrics $g_B$ and $g_F$, respectively, and let
$f$ be a positive function on $B$. Consider the product manifold $B\times F$ with
its projection $\pi:B\times F\to B$ and $\eta:B\times F\to F$. The {\it warped product}
$M=B\times_f F$ is the manifold $B\times F$ equipped with the Riemannian
structure such that 
\begin{equation}\label{E:warped} ||X||^2=||\pi_*(X)||^2+f^2(\pi(x))
||\eta_*(X)||^2\end{equation}

\noindent for any tangent vector $X\in T_xM$. Thus we have $g=g_B+f^2 g_F$. The function $f$ is called the {\it warping function\/}  (cf. \cite{bishop}). The concept of warped products appeared in the mathematical and physical literature long before \cite{bishop}, e.g., warped products were called semi-reducible spaces in \cite{Kr}.
It is well-known that the notion of warped products plays important roles in differential geometry as well as in physics, especially in the theory of general relativity (cf. \cite{c11,oneill}). 
 
According to a famous theorem of J. F. Nash, every Riemannian manifold can be isometrically
immersed in some Euclidean spaces with sufficiently high codimension.  In particular, Nash's embedding theorem implies that {\it every warped product manifold $N_1\times_f N_2$ can be isometrically embedded as Riemannian submanifolds in  Euclidean spaces with sufficiently high codimension.}  

In view of Nash's theorem, the author asked in early 2000s the following fundamental question concerning warped product submanifolds (see \cite{c02-2}).
\vskip.05in

{\bf Fundamental Question:} 
 {\it What can we conclude from an arbitrary isometric immersion of a warped product manifold into a Euclidean space with arbitrary codimension\,?} 
\vskip.03in

\vskip.03in
Or, more generally,
\vskip.03in

 {\it What can we conclude from an arbitrary isometric immersion of a warped product manifold into an arbitrary  Riemannian manifold\,?} \vskip.03in
 
 The study of this fundamental question on warped product submanifolds was not initiated until the beginning of this century by the author in a series of articles \cite{c5.1,c7,c8,c02-2}.  Since then the study of  warped product submanifolds has become a very active research topic in differential geometry of submanifolds.
 
In this article we survey the most important results on warped product submanifolds in various manifolds, including Riemannian,  K\"ahler, nearly K\"ahler, para-K\"ahler and Sasakian manifolds. 
It is the author's hope that this survey article will provide a good introduction on the theory of warped product submanifolds as well as a useful reference for further research on this subject.

\section{\uppercase{Preliminaries}} In this section we provide some basic notations, formulas, definitions, and results for later use.

\subsection{Basic notations and formulas} Let $M$ be an $n$-dimensional submanifold of a
 Riemannian $m$-manifold $\tilde M^m$.   We choose a local field of orthonormal
frame
$$e_1,\ldots,e_n,e_{n+1},\ldots,e_m$$ in $\tilde M^m$ such that, restricted to  $M$, the vectors $e_1,\ldots,e_n$ are tangent to $M$ and hence $e_{n+1},\ldots,e_m$ are normal to $M$. 
Let $K(e_i\wedge e_j)$ and $\tilde K(e_i\wedge e_j)$ denote respectively the sectional curvatures of $M$ and $\tilde M^m$ of the plane section spanned by $e_i$ and $e_j$. 

For the submanifold $M$ in  $\tilde M^m$ we denote by $\nabla$ and ${\tilde \nabla}$ the Levi-Civita connections of $M$ and $\tilde M^m$, respectively. The Gauss and Weingarten
formulas are given respectively by (see, for instance, \cite{cbook,c11})
\begin{align} &{\tilde \nabla}_{X}Y=\nabla_{X} Y +\sigma(X,Y),\\ &{\tilde\nabla}_{X}\xi =
A_{\xi}X+D_{X}\xi \end{align} 
for any  vector fields $X,Y$ tangent to $M$ and  vector field $\xi$ normal to $M$, where $\sigma$ denotes the second fundamental form, $D$ the normal connection, and $A$ the shape operator
of the submanifold. 

Let $\{\sigma^r_{ij}\}$, $i,j=1,\ldots,n;\,r=n+1,\ldots,m$,  denote the coefficients of the second fundamental form $h$ with respect to $e_1,\ldots,e_n,e_{n+1},\ldots,e_m$. Then we have $$\sigma^r_{ij}=\<\sigma(e_i,e_j),e_r\>=\<A_{e_r}e_i,e_j\>,$$ where $\<\;\,,\;\>$ denotes the inner product.

The mean curvature vector $\overrightarrow{H}$ is defined by
\begin{align}\overrightarrow{H} = {1\over n}\,\hbox{\rm trace}\,\sigma = {1\over n}\sum_{i=1}^{n} \sigma(e_{i},e_{i}), \end{align}
where $\{e_{1},\ldots,e_{n}\}$ is a local  orthonormal frame of the tangent bundle $TM$ of $M$. The squared mean curvature is then given by $$H^2=\left<\right.\hskip-.02in \overrightarrow{H},\overrightarrow{H}\hskip-.02in\left.\right>.$$ A submanifold $M$  is called  minimal  in $\tilde M^m$ if  its mean curvature vector  vanishes identically. 

Denote by $R$ and $\tilde R$  the Riemann curvature tensors of $M$ and $\tilde M^m$, respectively. Then the {\it equation of Gauss\/} is given  by \begin{equation}\begin{aligned} &R(X,Y;Z,W)=\tilde R(X,Y;Z,W)+\<\sigma(X,W),\sigma(Y,Z)\> -\<\sigma(X,Z),\sigma(Y,W)\>\end{aligned}\end{equation} 
for vectors $X,Y,Z,W$ tangent to $M$. In particular, for a submanifold of a Riemannian manifold of constant sectional curvature $c$, we have 
\begin{equation}\begin{aligned} R(X&,Y;Z,W)=c\{\left<X,W\right>\left<Y,Z\right>-
\left<X,Z\right>\left<Y,W\right>\}\\ & +\left<\sigma(X,W),\sigma(Y,Z)\right>
-\left<\sigma(X,Z),\sigma(Y,W)\right>.\end{aligned}\end{equation}

Let $M$ be a Riemannian $p$-manifold and  $e_1,\ldots,e_p$ be an orthonormal frame fields on $M$. For differentiable function $\varphi$ on $M$, the Laplacian $\Delta\varphi$ of $\varphi$ is defined  by
\begin{equation} \Delta\varphi=\sum_{j=1}^p \{(\nabla_{e_j}e_j)\varphi
-e_je_j\varphi\}.\end{equation}

For any orthonormal basis $e_1,\ldots,e_n$ of the tangent space $T_pM$, the scalar curvature $\tau$ of $M$ at $p$ is defined to be \begin{align}\tau(p)=\sum_{i<j} K(e_i\wedge e_j),\end{align} where $K(e_i\wedge e_j)$ denotes the sectional curvature of the plane section spanned by $e_i$ and $e_j$.

Let $L$ be a subspace of $T_pM$  of dimension $r\geq 2$  and $\{e_1,\ldots,e_r\}$ an orthonormal basis of $L$. The scalar curvature $\tau(L)$ of the $r$-plane section $L$ is defined by 
\begin{align}\label{2.6}\tau(L)=\sum_{\alpha<\beta} K(e_\alpha\wedge
e_\beta),\quad 1\leq \alpha,\beta\leq r.\end{align} 

\subsection{$\delta$-invariants}
For a  Riemannian $n$-manifold $M$, let $K(\psi)$ denote the sectional curvature associated with a 2-plane section $\psi\subset T_xM,\,x\in M$. For an orthonormal basis $\{e_1,\ldots,e_n\}$ of $T_xM$, the scalar curvature $\tau_{M}$ of $M$ at $x$ is defined to be \begin{align}\tau_{M}(x)=\sum_{i<j} K(e_i, e_j).\end{align}

Let $L$ be a $r$-subspace of $T_xM$ with $r\geq 2$  and $\{e_1,\ldots,e_r\}$ an orthonormal basis of $L$. We define the scalar curvature $\tau(L)$ of $L$ by  
$$\tau(L)=\sum_{\alpha<\beta} K(e_\alpha, e_\beta),\;\; 1\leq\alpha,\beta\leq r.$$

For an integer $k\geq 0$ let  ${\mathcal S}(n,k)$ denote the set consisting of unordered $k$-tuples $(n_1,\ldots,n_k)$ of integers $\geq 2$ such that  $ n>n_1$ and $ n_1+\cdots+n_k\leq n.$ Denote by ${\mathcal S}(n)$ the set of unordered $k$-tuples with $k\geq 0$ for a fixed $n$. 

 For each $k$-tuple $(n_1,\ldots,n_k)\in {\mathcal S}(n)$, the $\delta$-invariant $\delta{(n_1,\ldots,n_k)}(x)$ is defined by 
\begin{align}\notag \delta (n_1,\ldots,n_k)(x)=\tau_{M}(x)-\inf\{\tau(L_1)+\cdots+\tau(L_k)\}, \end{align}
where $L_1,\ldots,L_k$ run over all $k$ mutually orthogonal subspaces of $T_xM$ such that  $\dim L_j=n_j,\, j=1,\ldots,k$.

Some other invariants of similar nature, i.e., invariants obtained from the scalar curvature by deleting certain amount of sectional curvatures,  are also called $\delta$-{\it invariants}. For a general survey on $\delta$-invariants and their applications, see the recent book \cite{c11}.

\subsection{Complex extensors}
We recall the notions of complex extensors and Lagrangian $H$-umbilical submanifolds introduced in \cite{c4.1}.  

Let $G:N^{p-1}\rightarrow \mathbb E^p$ be an isometric immersion of a Riemannian $(p-1)$-manifold into the Euclidean $p$-space $\mathbb E^p$ and let $F:I\rightarrow {\mathbb C}^*$ be a unit speed curve in  ${\mathbb C}^*={\mathbb C}-\{0\}$. We extend $G:N^{p-1}\rightarrow \mathbb E^p$ to an immersion of $I\times
N^{p-1}$ into  ${\mathbb C}^p$  as
 \begin{align}F\otimes G:I\times N^{p-1}\rightarrow {\mathbb C}\otimes \mathbb E^p={\mathbb C}^p,\end{align}
where
$(F\otimes G)(s,q)=F(s)\otimes G(q)$ for $ s\in I, \; q\in N^{p-1}.$ This extension $F\otimes G$ of $G$ via tensor product is called the {\it complex extensor\/} of $G$ via $F$.

 A Lagrangian submanifold of $\tilde M^p$ without totally geodesic points is called {\it $H$-umbilical\/} if its second
fundamental form takes the following simple form (cf. \cite{c4.1}):  \begin{equation} \begin{aligned} \label{2.12}&
\sigma(\bar e_1,\bar e_1)\hskip-.01in =\hskip-.01in  \lambda J\bar e_1,\; \sigma(\bar e_j,\bar e_j)\hskip-.01in =\hskip-.01in \mu
J\bar e_1, \; j>1,\\ & \sigma(\bar e_1,\bar e_j)\hskip-.01in =\hskip-.01in \mu  J\bar e_j,\, \sigma(\bar e_j,\bar e_k)\hskip-.01in =\hskip-.01in 0,\, 2\leq j\ne k\leq p \end{aligned}\end{equation}
  for some  functions $\lambda,\mu$ with respect to a suitable
orthonormal local frame field $\{\bar e_1,\ldots,\bar e_p\}$.
Such submanifolds are the simplest Lagrangian submanifolds   next to the totally geodesic ones.

\begin{example} $($Lagrangian pseudo-sphere\,$)$ {\rm  For a real number $b>0$, let $F:{\mathbb R}\to {\mathbb C}$ be the unit speed curve given by
$$F(s)={{e^{2bsi}+1}\over{2bi}}.$$ 
With respect to the induced metric, the complex extensor $\phi=F\otimes \iota$ of the unit hypersphere of ${\mathbb E}^n$ via $F$ is a Lagrangian isometric immersion of an open portion of an $n$-sphere $S^n(b^2)$ of  sectional curvature $b^2$ into ${\mathbb C}^n$ which is simply called a {\it Lagrangian pseudo-sphere.}

A Lagrangian pseudo-sphere is a Lagrangian $H$-umbilical submanifold satisfying \eqref{2.12} with $\lambda=2\mu$. Conversely,  Lagrangian pseudo-spheres are the only Lagrangian $H$-umbilical submanifolds of ${\mathbb C}^n$ which satisfy \eqref{2.12}  with $\lambda=2\mu$.}\end{example}

\begin{example} $($Whitney sphere\,$)$ {\rm  Let $w:S^p(1)\rightarrow {\mathbb C}^p$ be the map of the unit $p$-sphere into ${\mathbb C}^{p}$ defined by
$$w(y_0,y_1,\ldots,y_p)={{1+{\rm i}y_0}\over {1+y_0^2}}( y_1,\ldots,y_p),\quad y_0^2+y_1^2+\ldots+y_p^2=1.$$
The map $w$ is a (non-isometric) Lagrangian immersion  with one self-intersection point which is called the {\it Whitney $p$-sphere.}  The Whitney
$p$-sphere is a complex extensor $\phi=F\otimes\iota$ of  $\iota:S^{p-1}(1) \subset \mathbb E^p$ via $F$, where $F=F(s)$
is an arclength reparametrization of the  curve  $f:I\rightarrow {\mathbb C}$ defined  by
$$f(t)={{\sin t+{\rm i}\sin t \cos t }\over{1+\cos^2 t}}.$$  

It is well-known that, up to rigid motions and dilations, the Whitney  sphere is the only Lagrangian $H$-umbilical submanifold in ${\mathbb C}^{p}$ which satisfies \eqref{2.12} with $\lambda=3\mu$.}\end{example}

The following results from \cite{c4.1} are fundamental for complex extensors.

\begin{proposition}\label{P:2.1}  Let $\iota:S^{p-1}\rightarrow \mathbb E^p$ be a hypersphere of $\mathbb E^p$ centered at the origin. Then every complex extensor $\phi=F\otimes \iota$ of $\iota$ via a unit speed curve $F:I\to \bf C^*$ is a Lagrangian $H$-umbilical submanifold of ${\mathbb C}^p$ unless $F$ is contained in a line through the origin $($which gives a totally geodesic Lagrangian submanifold$\,)$. \end{proposition}

\begin{theorem} Let $n\geq 3$ and $L: M\to {\mathbb C}^n$  be a Lagrangian $H$-umbilical isometric immersion. Then we have: 
\begin{enumerate}

\item If $M$ is of constant sectional curvature, then either $M$ is flat or, up to rigid motions of ${\mathbb C}^n$, $L$ is a Lagrangian pseudo-sphere.

\item If $M$ contains no open subset of constant sectional curvature, then,  up to rigid motions of ${\mathbb C}^n$, $L$ is a complex extensor of the unit hypersphere of ${\mathbb E}^n$.
\end{enumerate}\end{theorem}

\subsection{Warped product immersion}

Suppose that $M_1,\ldots,M_k$ are  Riemannian manifolds and that $$f:M_1\times\cdots\times M_k\to E^N$$ is an isometric immersion of the Riemannian product $M_1\times\cdots\times M_k$ into Euclidean $N$-space. J. D. Moore
\cite{moore}  proved that if the second fundamental form $\sigma$ of $f$ has the property that $\sigma(X,Y)=0$ for $X$  tangent to $M_i$ and $Y$ tangent to $M_j$, $i\ne j$, then $f$ is a product immersion, that is, there exist isometric immersions $f_i:M_i\to  E^{m_i},\, 1\leq i\leq k$, such that
\begin{equation}f(x_1,\ldots,x_k)=(f(x_1),\ldots,f(x_k))\end{equation}
when $x_i\in M_i$ for $1\leq i\leq k$.

 Let $M_0,\cdots ,M_k$ be Riemannian manifolds, $M=M_ 0 \times\cdots \times M_ k$ their product, and $\pi_i : M\to M_i$  the canonical projection. If $\rho_1,\cdots ,\rho_k : M_0\to \hbox{\bf R}_+$ are positive-valued functions, then 
\begin{equation}\left< X,Y \right>:= \left<\pi_{0*}X,\pi_{0*}Y\right> + \sum^k_{i=1} (\rho_i \circ\pi_0)^2 \left< \pi\sb {i*}X, \pi\sb {i*}Y\right>\end{equation} defines a Riemannian metric on $M$, called a warped product metric. $M$ endowed with this metric is denoted by $M_0 \times_{\rho_1} M_1 \times \cdots \times_{\rho\sb k}M_k$. 

A warped product immersion is defined as follows: Let $M_0 \times_{\rho_1}M_1\times \cdots\times_{\rho_k}M_k$ be a warped product and let $f_i: N_i\to M_i$, $i=0,\cdots,k$, be isometric immersions, and define $\sigma_i:= \rho_i \circ f_0 : N_0\to \hbox{\bf R}_+$ for $i = 1,\cdots ,k$.
Then the map \begin{equation}f: N_0 \times_{\sigma_1} N_1\times \cdots \times_{\sigma_k} N_k
\to M_0 \times_{\rho_1} M_1\times\cdots \times_{\rho_k} M_k\end{equation}
given by $f(x_0,\cdots ,x_k):= (f_0(x_0),f_1(x_1),\cdots , f_k(x_k))$ is an isometric immersion, which is called a {\it warped product immersion.} 

S. N\"olker \cite{nolker} extended Moore's result to the following.
 \vskip.1in 
 
\begin{theorem} Let $f: N_0 \times_{\sigma_1} N_1 \times \cdots \times_{\sigma_k} N_k\to R^N(c)$ be an isometric immersion into a Riemannian manifold of constant curvature $c$. If $h$ is the second fundamental form of $f$ and
$h(X_i,X_j)=0$, for all vector fields $X_i$ and $X\sb j$, tangent to $N_i$ and $N_j$
respectively, with $i\ne j$, then, locally, $f$ is a warped product immersion. \end{theorem}

\vskip.1in
 
\section{\uppercase{Warped product submanifolds of Riemannian manifolds}}

In this section, we present solutions from \cite{c02} to the Fundamental Question mention in Introduction.
 
For a warped product $N_1\times_f N_2$,  we denote by $\mathcal D_1$ and $\mathcal D_2$ the distributions given by the vectors tangent to leaves and fibers, respectively. Thus  $\mathcal D_1$ is obtained from tangent vectors of $N_1$ via the horizontal lift and  $\mathcal D_2$ obtained by tangent vectors of $N_2$ via
the vertical lift.

 Let $\phi:N_1\times_f N_2\to R^m(c)$ be an isometric immersion of a warped product $N_1\times_f N_2$ into a Riemannian manifold with  constant sectional curvature $c$. Denote by $\sigma$ the second fundamental form of $\phi$. 
 
 The immersion $\phi:N_1\times_f N_2\to R^m(c)$ is called {\it mixed totally geodesic\/} if $\sigma(X,Z)=0$ for any $X$ in $\mathcal D_1$ and $Z$ in  $\mathcal D_2$.

The following result was proved in \cite{c02}.

\begin{theorem}\label{T:3.1} For any  isometric immersion \ $\phi:N_1\times_f N_2\to R^m(c)$ of a warped product $\,N_1\times_f N_2$  into a  Riemannian manifold of
constant curvature $c$, we have
\begin{equation}\label{E:wpfirst} \frac{\Delta f} {f}\leq \frac{n^2}{4n_2}H^2+ n_1c,\end{equation}
 where $n_i=\dim N_i$, $n=n_1+n_2$,  $H^2$ is the squared mean curvature of $\phi$, and $\Delta f$ is the Laplacian $f$ on  $N_1$. 

 The equality sign of \eqref{E:wpfirst}  holds identically if and only if $\iota :N_1\times_f N_2\to R^m(c)$ is a mixed totally geodesic immersion satisfying ${\rm trace}\, h_1={\rm trace}\,h_2$, where ${\rm trace}\, h_1$ and ${\rm trace }\,h_2$ denote the trace of $\sigma$ restricted to $N_1$ and $N_2$, respectively.\end{theorem}  

The classification of isometric immersions from warped products into a real space form $R^m(c)$ satisfying the equality case of  \eqref{E:wpfirst} have been obtained in  \cite{c05-2}.

Theorem \ref{T:3.1} has many applications \cite{c02}. 

\begin{corollary}\label{C:1}  Let $N_1$ and $N_2$ be two Riemannian manifolds and $f$ be a nonzero harmonic  function on  $N_1$. Then every  minimal isometric immersion of $N_1\times_f N_2$ into any Euclidean space is a warped product immersion.
\end{corollary}

\begin{remark} {\rm There exist many minimal immersion of $N_1\times_f N_2$ into Euclidean spaces. For example, if $N_2$ is a minimal submanifold of $S^{m-1}(1)\subset\mathbb E^m$, then the minimal cone $C(N_2)$ over $N_2$ with vertex at the origin is a warped product $\mathbb R_+\times_s N_2$ whose warping function $f=s$ is a harmonic function. Here $s$ is the coordinate function of the positive real line $\mathbb R_+$.}\end{remark} 

\begin{corollary} \label{C:2} Let $f\ne 0$ be a harmonic function on  $N_1$. Then for any Riemannian manifold $N_2$ the warped product $N_1\times_f N_2$ does not admits any  minimal isometric immersion into any hyperbolic space. \end{corollary}

\begin{corollary} \label{C:3} Let $f$ be a function on $N_1$ with $\Delta f=\lambda f$ with $\lambda>0$. Then for any Riemannian manifold $N_2$ the warped product $N_1\times_f N_2$ admits no  minimal isometric immersion into any Euclidean space or   hyperbolic space.
\end{corollary}

\begin{remark} {\rm In views of  Corollaries \ref{C:2} and \ref{C:3}, we point out that {there exist  minimal immersions from $N_1\times_f N_2$  into hyperbolic space such that the warping function $f$ is an eigenfunction with negative eigenvalue}. For example, $\hbox{\bf R}\times_{e^x} \mathbb E^{n-1}$ admits an isometric minimal immersion into the hyperbolic space $H^{n+1}(-1)$ of constant  curvature $-1$.}
\end{remark}  

\begin{corollary} \label{C:4}  If $N_1$ is a compact, then  $N_1\times_f N_2$ does not admit a minimal isometric immersion into any Euclidean space or hyperbolic space. \end{corollary} 

\begin{remark} {\rm For Corollary \ref{C:4}, we  point out that {there exist many  minimal immersions of $N_1\times_f N_2$ into $\mathbb E^m$ with compact $N_2$}. For examples, a hypercaternoid in $\mathbb E^{n+1}$ is a minimal hypersurfaces which is isometric to a warped product $\mathbb R\times_f
S^{n-1}$. Also, for any compact minimal submanifold $N_2$ of $S^{m-1}\subset\mathbb E^m$, the minimal cone $C(N_2)$ is a warped product $\mathbb
R_+\times_s N_2$ which is also a such example.} \end{remark} 

\begin{remark} {\rm Contrast to Euclidean and hyperbolic spaces,  $S^{m}$ {admits minimal warped product submanifolds $N_1\times_f N_2$ with  $N_1$ and $N_2$ being compact}. The simplest such examples are minimal Clifford tori defined by
$$M_{k,n-k}=S^k\Big(\sqrt{\tfrac{k}{n}}\,\Big)\times S^{n-k}\Big(\sqrt{\tfrac{n-k}{n}}\,\Big)\subset S^{n+1},\; 1\leq k<n.$$}
\end{remark} 

\begin{remark} {\rm  Suceav\u{a} constructs in \cite{suceava}  a family of warped products of hyperbolic planes which do not admit any isometric minimal immersion into Euclidean space, by applying some $\delta$-invariants introduced in \cite{c5}.}   \end{remark} 
   
By  making a minor modification of the proof of Theorem \ref{T:3.1} in \cite{c02}, using the method of \cite{c05}, we have the following  general solution from \cite{CWei} to the Fundamental Question.

\begin{theorem}\label{T:3.2} If $\tilde M^m_c$ is a Riemannian manifold with sectional curvatures bounded from above by a constant $c$, then for any isometric immersion $\phi : N_1 \times _f N_2 \to \tilde M^m_c$   from a warped product $N_1 \times _f N_2$ into  $\tilde M^m_c$ the warping function $f$ satisfies
\begin{align}\label{3.1} \frac{\Delta f} {f}\leq \frac{n^2}{4n_2}H^2+ n_1c,\end{align}
where $n_1= \dim N_1$ and $n_2=\dim N_2$.  \end{theorem}

Another general solution to the Fundamental Question is the following result  from \cite{c04}.

\begin{theorem}[\cite{c12}]  For any  isometric immersion $\phi:N_1\times_f N_2\to R^m(c)$, the scalar curvature $\tau$ of  the warped product $N_1\times_f N_2$ satisfies
\begin{equation}\label{E:3.2}\tau\leq {\Delta f\over {n_1f}}+{{n^2(n-2)}\over{2(n-1)}}H^2+{1\over 2}(n+1)(n-2)c.\end{equation} 

 If $n=2$, the equality case of  \eqref{E:3.2} holds automatically.  

 If $n\geq 3$, the equality sign of  \eqref{E:3.2} holds identically if and only if  either

\begin{enumerate}
 \item  $N_1\times_f N_2$  is of constant  curvature $c$, the warping function $f$ is an eigenfunction with eigenvalue $c$, i.e.,  $\Delta f=cf$, and $N_1\times_f N_2$ is immersed as a totally geodesic submanifold in $R^m(c)$, or 

\item locally, $\,N_1\times_f N_2$ is immersed as a rotational hypersurface in a  totally geodesic submanifold $R^{n+1}(c)$ of $R^m(c)$ with a geodesic of $R^{n+1}(c)$ as its profile curve.\end{enumerate}\end{theorem}

\begin{remark} {\rm Every Riemannian manifold of constant  curvature  $c$ can be locally expressed as a warped product  whose warping function satisfies $\Delta f=cf$.  For examples,  $S^n(1)$ is locally isometric to $(0,\infty)\times_{\cos x} S^{n-1}(1)$;  $\mathbb E^n$ is locally isometric to $(0,\infty)\times_x S^{n-1}(1)$;
and  $H^n(-1)$ is locally isometric to $\mathbb R\times_{e^x} \mathbb E^{n-1}$.

There are other warped product decompositions of $R^n(c)$ whose
warping function satisfies $\Delta f=cf$.  For example, let $\{x_1,\ldots,x_{n_1}\}$ be a Euclidean coordinate system of $\mathbb E^{n_1}$ and let
$\rho=\sum_{j=1}^{n_1} a_j x_j+b$, where $a_1,\ldots,a_{n_1},b$ are real numbers satisfying $\sum_{j=1}^{n_1} a_j^2=1$. Then the warped product $\mathbb E^{n_1}\times_\rho S^{n_2}(1)$ is a flat space whose warping function is a harmonic function. In fact, those are the only warped product decompositions of flat spaces  whose warping functions are harmonic functions.} \end{remark}

\vskip.1in

\section{\uppercase{Multiply warped product submanifolds}}

 Let $(N_i,g_{i}), i=1,\ldots,k,$ be Riemannian manifolds.
For a multiply warped product manifold $N_1 \times_{f_2} N_2 \times\cdots \times_{f_k}N_k$,  let $\mathcal D_i$  denote the distributions obtained from the vectors tangent to  $N_i$ (or more precisely, vectors tangent to the horizontal lifts of $N_i$). Assume that  $$\phi:N_1 \times_{f_2} N_2 \times \cdots \times_{f_k}N_k\to \tilde M$$ is an isometric immersion of a multiply warped product $N_1 \times_{f_2} N_2 \times \cdots \times_{f_k}N_k$ into a Riemannian manifold $\tilde M$. Denote by $\sigma$ the second fundamental form of $\phi$. Then the immersion $\phi$ is called {\it mixed totally geodesic\/} if $\sigma(\mathcal D_i,\mathcal D_j)=\{0\}$ holds for  distinct $i,j\in \{1,\ldots,k\}$.

Let $\psi:N_1 \times_{f_2} N_2 \times \cdots \times_{f_k} N_k\to\tilde M$ be an isometric immersion of a multiply warped product $N_1 \times_{f_2} N_2 \times \cdots \times_{f_k} N_k$ into an arbitrary Riemannian manifold $\tilde M$.   Denote by  ${\rm trace}\,h_i$   the trace of $\sigma$ restricted to $N_i$,  that is
\begin{align*}{\rm trace}\,h_i=\sum_{\alpha=1}^{n_i} \sigma(e_\alpha,e_\alpha)\end{align*}
for some orthonormal frame fields $e_1,\ldots,e_{n_i}$  of $\mathcal D_i$.

An extension of Theorems \ref{T:3.1} and \ref{T:3.2} is the following result from \cite{CD08-2}.

\begin{theorem}\label{T:CD} Let $\phi:N_1 \times_{f_2} N_2 \times \cdots \times_{f_k} N_k\to \tilde M^m$ be an isometric immersion of a multiply warped product
$N:=N_1 \times_{f_2} N_2 \times \cdots \times_{f_k} N_k$ into an arbitrary Riemannian $m$-manifold. Then we have
\begin{align} \label{01.3} & \sum_{j=2}^k n_j\frac{\Delta f_j}{f_j} \leq \frac{n^2(k-1)}{2k} H^2
 + n_1(n-n_1) \max\tilde K, \end{align}
where $n=\sum_{i=1}^{k}n_{i}$ and $\max\tilde K(p)$ denotes the maximum of the sectional curvature function of $\tilde M^m$ restricted to 2-plane sections of the tangent space $T_pN$ of $N$ at $p=(p_1,\ldots,p_k)$.

The equality sign of \eqref{01.3} holds identically if and only if the following two statements hold:
\begin{enumerate}
 \item $\phi$ is a mixed totally geodesic immersion satisfying ${\rm trace}\, h_1=\cdots={\rm trace }\,h_k$;

\item at each point $p\in N$,  the sectional curvature function $\tilde K$ of $\tilde M^m$ satisfies $$\tilde K(u,v)=\max \tilde K(p)$$ for each unit vector $u$ in $T_{p_1}(N_1)$ and unit vector $v$ in $T_{(p_2,\cdots,p_k)}(N_2\times \cdots\times N_k)$. \end{enumerate} \end{theorem}

The following example shows that inequalities \eqref{01.3} is sharp.

\begin{example}\label{E:4.1} {\rm
Let $M_1\times_{\rho_2} M_2\times \cdots\times_{\rho_k} M_k$ be a {\it multiply warped product representation\/} of a real space form $R^m(c)$. Assume that $\psi_1:N_1\to M_1$ is a minimal immersion of $N_1$ into $M_1$ and let $f_2,\ldots,f_k$ be the restrictions of $\rho_2,\ldots,\rho_k$ on $N_1$. Then the following warped product
immersion:
$$\psi=(\psi_1,id,\ldots,id):N_1\times_{f_2} M_2\times \cdots\times_{f_k} M_k\to
M_1\times_{\rho_2} M_2\times \cdots\times_{\rho_k} M_k\subset R^m(c)$$
is a mixed totally geodesic warped product submanifold of $R^m(c)$ which satisfies the condition:
\begin{align*}&{\rm trace}\, h_1=\cdots={\rm trace }\,h_k =0. \end{align*}
Thus $\psi$ satisfies the equality case of \eqref{01.3} according to Theorem \ref{T:CD}. Therefore inequality \eqref{01.3} is optimal.}
\end{example}

By applying Theorem \ref{T:CD} we have  the following corollaries.
 
\begin{corollary}\label{C:4.1} Let $\phi:N_1 \times_{f_2} N_2 \times \cdots \times_{f_k} N_k\to R^m(c)$ be an isometric immersion of the multiply warped product $N_1 \times_{f_2} N_2 \times \cdots \times_{f_k} N_k$ into a Riemannian $m$-manifold of constant curvature $c$.  If we have
 \begin{align*}  & \sum_{j=2}^k n_j\frac{\Delta f_j}{f_j} = \frac{n^2(k-1)}{2k} H^2 + n_1(n-n_1) c, \end{align*}
 then $\phi$ is a warped product immersion.
 \end{corollary}

\begin{corollary}\label{C:4.2} If $f_2,\ldots,f_k$ are  harmonic functions on $N_1$ or
eigenfunctions of the Laplacian on $N_1$ with positive eigenvalues, then the multiply warped product  $N_1 \times_{f_2} N_2 \times \cdots \times_{f_k} N_k$  cannot be isometrically immersed into every Riemannian manifold of negative sectional curvature as a minimal submanifold.
\end{corollary}

\begin{corollary}\label{C:4.3} If $f_2,\ldots,f_k$ are  eigenfunctions of the Laplacian $\Delta$ on $N_1$ with nonnegative eigenvalues and at least one of $f_2,\ldots, f_k$ is non-harmonic, then the multiply warped product manifold $N_1\times_{f_2} N_2 \times \cdots \times_{f_k} N_k$ cannot be isometrically immersed into every Riemannian manifold of non-positive sectional curvature as a minimal submanifold.
 \end{corollary}

\begin{corollary}\label{C:4.4} If $f_2,\ldots,f_k$ are harmonic functions on $N_1$, then every isometric minimal immersion of the multiply warped product manifold
$N_1 \times_{f_2} N_2 \times \cdots \times_{f_k} N_k$ into a Euclidean space is a warped product immersion. \end{corollary}

Let $(N_{1},g_{1})$ and $(N_{2},g_{2})$ be two Riemannian manifolds and let   $\sigma_{1}:N_{1}\to (0,\infty)$ and $\sigma_{2}:N_{2}\to (0,\infty)$ be differentiable functions. The doubly warped product $N={}_{\sigma_{2}} N_{1}\times_{\sigma_{1}} N_{2} $ is the product manifold $N_{1}\times N_{2}$ endowed with the metric
$$ g=\sigma_{2}^{2}g_{1}+\sigma_{1}^{2} g_{2}.$$ 

The following result is obtained in \cite{Ol10-3}.

\begin{theorem}\label{C:4.5} Let $\phi:{}_{\sigma_{2}}N_1 \times_{\sigma_{1}} N_2\to \tilde M^m$ be an isometric immersion of a doubly warped product
${}_{\sigma_{2}}N_1 \times_{\sigma_{1}} N_2$ into an arbitrary Riemannian $m$-manifold. Then we have
\begin{align} \label{0.3} & n_{2}\frac{\Delta_{1}\sigma_{1}}{\sigma_{1}}+ n_1\frac{\Delta_{2} \sigma_{2}}{\sigma_{2}} \leq \frac{n^2}{4} H^2
 + n_1n_{2} \max\tilde K, \end{align}
where $n_{i}=\dim N_{i}, n=n_{1}+n_{2}$, $\Delta_{i}$ denotes the Laplacian of  $N_{i}$.

The equality sign of \eqref{0.3} holds identically if and only if the following two statements hold:
\begin{enumerate}
 \item $\phi$ is a mixed totally geodesic immersion satisfying ${\rm trace}\, h_1={\rm trace }\,h_2$ and

\item at each point $p=(p_{1},p_{2})\in N$,  the sectional curvature function $\tilde K$ of $\tilde M^m$ satisfies $\tilde K(u,v)=\max \tilde K(p)$ for each unit vector $u$ in $T_{p_1}(N_1)$ and unit vector $v$ in $T_{p_{2}}( N_2)$. \end{enumerate}
\end{theorem}

\vskip.15in

\section{\uppercase{Arbitrary warped products in complex space forms}}

Now, we present the following  results form \cite{c02-4,c03-2} for arbitrary  warped products submanifolds in non-flat complex space forms. 
  
For arbitrary  warped products submanifolds in complex hyperbolic spaces, we have the following general results from \cite{c02-4}.

\begin{theorem}\label{T:5.1} Let $\phi:N_1\times_f N_2\to CH^m(4c)$ be an arbitrary isometric immersion of a warped product $N_1\times_f N_2$ into the complex hyperbolic $m$-space $CH^m(4c)$ of constant holomorphic sectional curvature $4c$. Then we have
\begin{align} \label{E:1.1} {\Delta f\over {f}}\leq {{(n_1+n_2)^2}\over{4n_2}}H^2+ n_1c. \end{align} 

The equality sign of \eqref{E:1.1} holds if and only if the following three conditions hold.
\begin{enumerate}

\item  $\phi$ is mixed totally geodesic, 

\item ${\rm trace}\,h_1= {\rm trace}\,h_2$, and

\item $J{\mathcal D_1}\perp {\mathcal D_2}$, where $J$ is the almost complex structure of $CH^m$.
\end{enumerate}\end{theorem}

Some interesting applications of Theorem \ref{T:5.1} are the following three non-immersion theorems  \cite{c02-4}.

\begin{theorem}\label{T:5.2} Let $N_1\times_f N_2$ be a warped product whose warping function $f$ is  harmonic. Then  $N_1\times_f N_2$ does not admit an isometric minimal immersion into any complex hyperbolic space.
 \end{theorem}

\begin{theorem}\label{T:5.3} If $f$ is an eigenfunction of Laplacian on $N_1$ with eigenvalue $\lambda>0$, then  $N_1\times_f N_2$ does not admits an isometric minimal immersion into any complex hyperbolic space.
\end{theorem}

\begin{theorem}\label{T:5.4}If $N_1$ is compact, then every warped product $N_1\times_f N_2$ does not admit an isometric minimal immersion into any complex hyperbolic space.
\end{theorem}

For arbitrary  warped products submanifolds in the complex projective $m$-space $CP^{m}(4c)$ with constant holomorphic sectional curvature $4c$, we have the following results from \cite{c03-2}.

\begin{theorem}\label{T:5.5} Let $\phi:N_1\times_f N_2\to CP^m(4c)$ be an arbitrary isometric immersion of a warped product into the complex projective $m$-space $CP^m(4c)$ of constant holomorphic sectional curvature $4c$. Then we have 
\begin{align} \label{F:1.2}{\Delta f\over{f}}\leq {{(n_1+n_2)^2}\over{4n_2}}H^2+ (3+n_1)c.\end{align} 

The equality sign of \eqref{F:1.2} holds identically if and only if  we have 

\begin{enumerate}

\item  $n_1=n_2=1$,  

\item  $f$ is an eigenfunction of the Laplacian of $N_1$ with eigenvalue $4c$, and  

\item $\phi$ is totally geodesic and holomorphic.
\end{enumerate}
\end{theorem}
 
 From now on, we denote $S^{n}(1),\, RP^{n}(1)$, $CP^{n}(4)$ and $ CH^{n}(-4)$ simply  by $S^{n},\, RP^{n}$, $CP^{n}$ and $CH^{n}$, respectively.
 
An immediate  application of Theorem \ref{T:5.5} is the following non-immersion result.

\begin{theorem} \label{T:5.6}  If $f$ is a positive function on a Riemannian $n_1$-manifold $N_1$ such that $(\Delta f)/f>3+n_1$ at some point  $p\in N_1$, then, for any Riemannian manifold $N_2$, the warped product  $N_1\times_f N_2$ does not admit any isometric minimal immersion into  $CP^m$ for any $m$.
\end{theorem}

For totally real minimal immersions, Theorem \ref{T:5.6} can be sharpen as the following.

\begin{theorem} \label{T:5.7}  If $f$ is a positive function on a Riemannian $n_1$-manifold $N_1$ such that $(\Delta f)/f>n_1$ at some point  $p\in N_1$, then, for any Riemannian manifold $N_2$, the warped product  $N_1\times_f N_2$ does not admit any isometric totally real minimal immersion into  $CP^m$ for any $m$.\end{theorem}

The following examples show that Theorems \ref{T:5.5}, \ref{T:5.6} and \ref{T:5.7} are sharp.

\begin{example} {\rm Let $I=(-\frac{\pi}{4},\frac{\pi}{4}),\, N_2=S^1$ and $f=\frac{1}{2}\cos 2 s$. Then the warped product
$$N_1\times_f N_2=: I\times_{(\cos 2 s)/2}S^1 $$ has constant sectional curvature 4. Clearly, we have $(\Delta f)/f=4$. If we define the complex structure $J$ on the warped product by
$J\left({\partial\over {\partial s}}\right)=2(\sec 2 s) {\partial\over {\partial t}},$
 then $\,(I\times_{(\cos 2 s)/2}S^1,g,J)\,$ is  holomorphically isometric to a dense open subset of $CP^1$.

Let $\phi:CP^1\to CP^m$ be a standard totally geodesic embedding of $CP^1$ into $CP^m$. Then the restriction of $\phi$ to $I\times_{(\cos 2 s)/2}S^1$ gives rise to a minimal isometric immersion of $I\times_{(\cos 2 s)/2}
S^1$ into $CP^m$ which  satisfies the equality case of \eqref{F:1.2} on $I\times_{(\cos 2 s)/2} S^1$  identically. }
\end{example}

\begin{example} {\rm Consider the same warped product $N_1\times_f N_2=I\times_{(\cos 2 s)/2}S^1$ as given in Example 5.1. Let $\phi:CP^1\to CP^m$ be the  totally geodesic holomorphic embedding of $CP^1$ into $CP^m$. Then the restriction of $\phi$ to $N_1\times_f N_2$ is an isometric minimal immersion of $N_1\times_f N_2$ into $CP^m$ which satisfies $(\Delta f)/f=3+n_1$ identically. 
This example shows that the assumption ``$(\Delta f)/f>3+n_1$ at some point in $N_1$'' given in Theorem \ref{T:5.6} is best possible.}\end{example}

\begin{example} {\rm Let  $g_1$ be the standard metric on $S^{n-1}$. Denote by $N_1\times_f N_2$ the warped product given by $N_1=(-\pi/2,\pi/2),\,N_2=S^{n-1}$ and $f=\cos s$. Then the warping function of this warped product satisfies $\Delta f/{f}=n_1$
identically. Moreover, it is easy to verify that this warped product is isometric to a dense open subset of $S^n$.

Let \begin{align}\phi\,:\, & S^n\xrightarrow[\text{2:1}]{\text{projection}}
RP^{n} \xrightarrow[\text{totally real}]{\text{totally geodesic}}  CP^n\notag\end{align}
be a standard totally geodesic Lagrangian immersion of $S^n$ into $CP^n$. Then the restriction of $\phi$ to $N_1\times_f N_2$ is a totally real minimal
immersion.
This example illustrates that the assumption ``$(\Delta f)/f>n_1$ at some point in $N_1$'' given in Theorem \ref{T:5.7} is also sharp.}\end{example}

A {\it generalized complex space form} of constant curvature $c$ and constant type $\alpha$, denoted by $M(c,\alpha)$, is an almost Hermitian manifold $(M,J,g)$ whose curvature tensor $R$ satisfies
\begin{equation}\begin{aligned} &\hskip.4in R(X,Y,Z,W)=\frac{1}{4}(c+3\alpha)\{g(X,Z)g(Y,W)-g(X,W)g(Y,Z)\}
\\&+\frac{1}{4}(c-\alpha)\{g(X,JZ)g(Y,JW)-g(X,JW)g(Y,JZ)+2g(X,JY)g(Z,JW)\}
\end{aligned}\end{equation}
for some $c$ and $\alpha$. The class of generalized complex space forms contains all complex space forms and the nearly K\"ahler manifold $S^{6}$.

For warped products submanifold of a generalized complex space form, the following result is obtained in  \cite{M04}.

\begin{theorem} \label{T:15.8}  Let  $N=N_{T}\times_{f} N_{\perp}$ a warped product submanifold of a generalized complex space form $ M(c,\alpha)$. Then

\begin{enumerate}
\item For $c\leq \alpha$, one has
\begin{equation}\label{5.24}\frac{\Delta f}{f} \geq \frac{n^{2}}{4n_{2}} H^{2} +n_{1}\frac{c+3\alpha}{4} , \end{equation}
where $n_{i}=\dim N_{i}$ and $n=n_{1}+n_{2}$.

\item For $c>\alpha$, one has
\begin{equation}\label{5.25}\frac{\Delta f}{f} \geq \frac{n^{2}}{4n_{2}} H^{2} +n_{1}\frac{c+3\alpha}{4}+3\frac{c-\alpha}{8}||P||^{2} , \end{equation}
where $PX$ denotes the part of $JX$ which is tangent to the submanifold.
\end{enumerate}\end{theorem}
The equality case of \eqref{5.24} and \eqref{5.25} were discussed in \cite{M04}.

 For $CR$-warped product submanifold of a generalized complex space form, one has the following result from  \cite{AAK09}.

\begin{theorem} \label{T:5.9}  Let  $N=N_{T}\times_{f} N_{\perp}$ a $CR$-warped product submanifold of a generalized complex space form $ M(c,\alpha)$. Then we have
$$ ||\sigma||^{2}\geq 2p \Big\{ ||\nabla (\ln f)||^{2} +\frac{1}{2} \Delta (\ln f)+\frac{h(c-\alpha)}{4}\Big\}, $$
where $\sigma$ is the second fundamental form) and $2h$ and $p$ are the real dimensions of $N_{T}$ and $N_{\perp}$, respectively.
\end{theorem}

\vskip.1in

\section{\uppercase{Warped products as K\"ahlerian submanifolds}}

Let $(z_0^i,\ldots,z_{\alpha_i}^i),$ $1\leq i\leq s,$ be homogeneous coordinates of $CP^{\alpha_i}$. Define a map:
\begin{equation}S_{\alpha_1\cdots \alpha_s}:CP^{\alpha_1}\times\cdots\times CP^{\alpha_s}\to CP^N,\quad N=\prod_{i=1}^s (\alpha_i+1)-1,\end{equation} 
which maps a point $((z_0^1,\ldots,z_{\alpha_1}^1),\ldots,(z_0^s,\ldots,z_{\alpha_s}^s))$ in $CP^{\alpha_1}\times\cdots\times CP^{\alpha_s}$ to the point $(z^1_{i_1}\cdots z^s_{i_j})_{1\leq i_1\leq \alpha_1,\ldots,1\leq i_s\leq \alpha_s}$ in $CP^N$. The map $S_{\alpha_1\cdots \alpha_s}$ is a K\"ahler embedding which is known as the {\it Segre embedding}. The Segre embedding was constructed by C. Segre in 1891.  
 
The following results from \cite{c2,CK} obtained in 1981 can be regarded as ``converse'' to Segre embedding constructed in 1891.  

\begin{theorem}\label{T:s1} Let $M_1^{\alpha_1},\ldots,M_s^{\alpha_s}$ be K\"ahler manifolds of
dimensions $\alpha_1,\ldots,$ $\alpha_s$, respectively. Then every holomorphically isometric immersion
$$f:M_1^{\alpha_1}\times\cdots\times M_s^{\alpha_s}\to CP^N,\quad N=\prod_{i=1}^s (\alpha_i+1)-1,$$ 
of $M_1^{\alpha_1}\times\cdots\times M_s^{\alpha_s}$ into $CP^N$   is locally  the Segre embedding, i.e., $M_1^{\alpha_1},\ldots,M_s^{\alpha_s}$ are open portions of $CP^{\alpha_1},\ldots, CP^{\alpha_s}$, respectively. Moreover,  $f$ is congruent to
the Segre embedding.\end{theorem}

Let $\bar \nabla^k\sigma$, $k=0,1,2,\cdots$, denote the $k$-th covariant derivative of the second fundamental form. Denoted by $||\bar\nabla^k\sigma||^2$ the squared norm of $\bar \nabla^k\sigma$. 

The following result was proved in  \cite{c2,CK}.
 
\begin{theorem}\label{T:s2} Let
$M_1^{\alpha_1}\times\cdots\times M_s^{\alpha_s}$ be a product K\"ahler submanifold of  $CP^N$. Then 
\begin{equation}\label{E:6.2} ||\bar \nabla^{k-2}\sigma||^2\geq k!\,2^k\sum_{i_1<\cdots<i_k}\alpha_1\cdots\alpha_k,\end{equation} for $k=2,3,\cdots.$

The equality sign of \eqref{E:6.2}  holds for some $k$ if and only if   $M_1^{\alpha_1},\ldots,M_s^{\alpha_s}$ are open parts of $CP^{\alpha_1},\ldots, CP^{\alpha_s}$, respectively, and the immersion is congruent to the Segre embedding.\end{theorem}

In particular, if $k=2$, Theorem \ref{T:s2} reduces to the following result of \cite{c2}.

\begin{theorem}\label{T:s3} Let $M_1^{h}\times M_2^{p}$ be a product K\"ahler submanifold of  $CP^N$. Then we have
\begin{equation}\label{E:6.0} ||\sigma||^2\geq 8hp.\end{equation}

The equality sign of \eqref{E:6.0}  holds if and only if  $M_1^h$ and $M_2^p$ are open portions of $CP^{h}$ and $CP^{p}$, respectively, and the immersion is congruent to the Segre embedding $S_{h,p}$.
\end{theorem}

We may extend Theorem \ref{T:s3} to the following for warped products.

\begin{theorem}\label{T:s4} Let $(M_1^{h},g_{1})$ and $(M_2^{p},g_{2})$ be two K\"ahler manifolds of complex dimension $h$ and $p$ respectively and let $f$ be a positive function on $M_{1}^{h}$. If $\phi:M_1^{h}\times_{f} M_2^{p}\to CP^{N}$ is a holomorphically isometric immersion of the warped product manifold $M_1^{h}\times_{f} M_2^{p}$ into   $CP^N$. Then $f$ is a constant, say $c$. Moreover, we have
\begin{equation}\label{E:6.01} ||\sigma||^2\geq 8hp.\end{equation}

The equality sign of \eqref{E:6.01}  holds if and only if  $(M_1^h,g_{1})$ and $(M_2^p,cg_{2})$ are open portions of $CP^{h}$ and $CP^{p}$, respectively, and the immersion $\phi$ is congruent to the Segre embedding.
\end{theorem}
\begin{proof} Under the hypothesis, the warped product manifold $M_1^{h}\times_{f} M_2^{p}$ is a K\"ahler manifold. Therefore, the warping function must be a positive constant. Now, the theorem follows from Theorem \ref{T:s3}. 
\end{proof}

\vskip.1in

\section{\uppercase{$CR$-products in K\"ahler manifolds}}
 
 A submanifold $N$ in a K\"ahler manifold $\tilde M$ is called a {\it totally real submanifold} if the almost complex structure $J$ of $\tilde M$ carries each tangent space $T_xN$ of $N$ into its corresponding normal space $T_x^\perp N$  \cite{CO}. The submanifold $N$ is called a
holomorphic submanifold (or K\"ahler submanifold) if $J$ carries each $T_xN$ into itself.

A submanifold $N$ of a K\"ahler manifold $\tilde M$ is called a {\it $CR$-submanifold} \cite{bejancu} if there exists on $N$ a holomorphic distribution $\mathcal D$ whose orthogonal  complement $\mathcal D^\perp$ is a totally real distribution, i.e., $J\mathcal D^\perp_x\subset T^\perp_x N$. 

A $CR$-submanifold  of a K\"ahler manifold $\tilde M$ is called a
{\it $CR$-product} \cite{c2} if it is a Riemannian product  $N_T\times N_\perp$ of a K\"ahler submanifold $N_T$ and a totally real submanifold $N_\perp$.  It is called {\it mixed totally geodesic} if the second fundamental form of the $CR$-submanifold satisfying $\sigma(X,Z)=0$ for any $X\in \mathcal D$ and $Z\in \mathcal D^{\perp}$.

For $CR$-products in complex space forms, the following result from \cite{c2} are known. 

\begin{theorem}\label{T:7.1} We have
\begin{enumerate}
\item[(i)] A $CR$-submanifold in the complex Euclidean $m$-space $\mathbb C^m$  is a $CR$-product if and only if it is a direct sum of a K\"ahler submanifold and a totally real submanifold of  linear complex subspaces. 

\item[(ii)] There do not exist $CR$-products in complex hyperbolic spaces other than K\"ahler submanifolds and totally real submanifolds.
\end{enumerate}\end{theorem}

 $CR$-products $N_T\times N_\perp$ in $CP^{h+p+hp}$ are always obtained from the Segre embedding $S_{h,p}$; namely, we have the following results from \cite{c2}.

\begin{theorem}\label{T:7.2}  Let $N_T^{h}\times N_\perp^p$ be the $CR$-product in $CP^m$ with constant holomorphic sectional curvature $4$. Then 
\begin{equation}\label{E:7.a} m\geq h+p+ hp. \end{equation}
 
The equality sign of \eqref{E:7.a} holds if and only if 

\begin{enumerate}
\item[(a)] $N_T^h$ is a totally geodesic K\"ahler submanifold, 

\item[(b)] $N_\perp^p$ is a totally real submanifold, and 

\item[(c)] the immersion is given by
\begin{equation}\notag N_T^h\times N_\perp^p \xrightarrow{\text{{\rm}}} CP^h\times CP^p\xrightarrow [\text{{\rm Segre imbedding}}]{\text{$S_{hp}$}} CP^{h+p+hp}.\end{equation} 
\end{enumerate}\end{theorem}  

\begin{theorem} \label{T:7.3}  Let $N_T^{h}\times N_\perp^p$ be the $CR$-product in $CP^m$. Then the squared norm of the second fundamental form  satisfies
\begin{equation}\label{E:7.b} ||\sigma||^2\geq 4 hp.\end{equation}
 
The equality sign of \eqref{E:7.b} holds if and only if 

\begin{enumerate}
\item[(a)] $N_T^h$ is a totally geodesic K\"ahler submanifold, 

\item[(b)] $N_\perp^p$ is a totally geodesic totally real submanifold, and 

\item[(c)] the immersion is given by
\begin{equation}\notag N_T^h\times N_\perp^p \xrightarrow{\text{{\rm totally geodesic}}} CP^h\times CP^p\xrightarrow [\text{{\rm Segre imbedding}}]{\text{$S_{hp}$}} CP^{h+p+hp}\subset CP^m.
\end{equation} 
\end{enumerate}\end{theorem}

\vskip.1in

\section{\uppercase{Warped product Lagrangian submanifolds of K\"ahler manifolds}}

A totally real submanifold $N$ in a K\"ahler manifold $\tilde M$ is called a {\it Lagrangian submanifold} if $\dim_{\mathbb R}N=\dim_{\mathbb C}\tilde M$. For the most recent survey on differential geometry of Lagrangian submanifolds, see \cite{c6,c8.1}.

  For Lagrangian immersions into  complex Euclidean $n$-space $\mathbb C^n$, a well-known result of M. Gromov \cite{gromov} states that a compact $n$-manifold $M$ admits a Lagrangian immersion (not necessary isometric) into $\mathbb C^n$ if and only if the complexification $TM\otimes {\mathbb C}$ of the tangent bundle of $M$ is trivial. In particular, Gromov's result implies that there exists no topological obstruction to Lagrangian immersions for compact 3-manifolds in ${\mathbb C}^3$, because the tangent bundle of a 3-manifold is always trivial.  

Not every warped product $N_1\times_f N_2$ can be isometrically immersed in a complex space form as a Lagrangian submanifold. Therefore, from Riemannian point of view, it is natural to ask the following basic question.

\vskip.08in 
\noindent {\bf Problem 8.1.} {\it When a warped product $n$-manifold admits a
Lagrangian isometric immersion into $\mathbb C^n$?}
\vskip.08in 
  
 For warped products of  curves and the unit $(n-1)$-sphere $S^{n-1}$, we have the following existence theorem from \cite{c4.1,c6.1}.

\begin{theorem}\label{T:8.1} Every simply-connected open portion of a warped product manifold $I\times_f S^{n-1}$ of an open interval $I$ and a unit $(n-1)$-sphere admits an isometric Lagrangian immersion into $\mathbb C^n$. \end{theorem} 

The Lagrangian immersions given in Theorem \ref{T:8.1} are expressed in terms of complex extensors in the sense of \cite{c4.1}. In particular, Theorem \ref{T:8.1} implies the following.

\begin{corollary} Every warped product surface does admit a  Lagrangian isometric immersion into ${\mathbb C}^{2}$.
\end{corollary}

 Consequently, we have a complete solution to Problem 8.1 for $n=2$.

Since all rotation hypersurfaces and real space forms can be  expressed, at least locally, as  warped products of some curves and a unit sphere, therefore Theorem \ref{T:8.1} implies  the following
\begin{corollary}\label{C:8.1} Every rotation  hypersurface of $\,\mathbb E^{n+1}$ can be isometrically immersed as Lagrangian submanifold in $\mathbb C^n$.\end{corollary}

\begin{corollary}\label{C:8.2}  Every Riemannian $n$-manifold of constant sectional curvature $c$ can be locally isometrically immersed  in $\Bbb C^n$ as a Lagrangian submanifold. \end{corollary}

\begin{remark} {\rm Not every Riemannian $n$-manifold of constant sectional curvature can be {\bf globally} isometrically immersed in $\Bbb C^n$ as a Lagrangian submanifold. For instance, it is known from \cite{c5}  that every compact Riemannian $n$-manifold with positive sectional curvature (or with positive Ricci curvature) does not admit any Lagrangian isometric immersion into $\mathbb C^n$.}
\end{remark} 

 \begin{remark} {\rm For further results on warped products Lagrangian submanifolds in complex space forms, see \cite{BRV}.}\end{remark}

\vskip.15in

\section{\uppercase{Warped Product $CR$-submanifolds of K\"ahler manifolds}}
 
In this section we present results concerning warped product $CR$-submanifold in an arbitrary K\"ahler manifold.
First, we mention the following result from \cite{c7}. 

\begin{theorem}\label{T:9.1} If  $\,N_\perp\times_f N_T$ is a warped product $CR$-submanifold of a K\"ahler manifold $\tilde M$ such that $N_\perp$ is a totally real and $N_T$  a K\"ahler submanifold of $\tilde M$, then it is a $CR$-product.\end{theorem}

Theorem \ref{T:9.1} shows that there does not exist warped product $CR$-submanifolds of the form $N_\perp\times_f N_T$  other than $CR$-products. So, we shall only consider warped product $CR$-submanifolds of the form: $N_T\times_f N_\perp$, {\it by reversing the two factors $N_T$ and $N_\perp$}. We simply call such  $CR$-submanifolds $CR$-{\it warped products} \cite{c7}.
  
$CR$-{\it warped products} are characterized in \cite{c7} as follows.

\begin{proposition}\label{P:9.1}  A proper $CR$-submanifold $M$ of a K\"ahler manifold $\tilde M$ is locally a $CR$-warped product if and only if the shape operator $A$ satisfies \begin{equation}A_{JZ}X=((JX)\mu)Z,\quad X\in \mathcal D,\quad Z\in \mathcal D^\perp,\end{equation}
for some  function $\mu$ on $M$ satisfying $W\mu=0,\,\forall W\in \mathcal D^\perp$. \end{proposition} 

A fundamental general result on $CR$-warped products in arbitrary K\"ahler manifolds is  the following  theorem from \cite{c7}.

\begin{theorem} \label{T:9.2}  Let $N_T\times_f N_\perp$ be a  $CR$-warped product submanifold in an arbitrary K\"ahler manifold $\tilde M$. Then the second fundamental form $\sigma$ satisfies
\begin{equation}\label{E:9.1}||\sigma||^2\geq 2p\,||\nabla(\ln f)||^2,\end{equation} where $\nabla (\ln f)$ is the gradient of $\,\ln f$ on $N_T$ and $p=\dim N_\perp$.

If the equality sign of \eqref{E:9.1}  holds identically, then $N_T$ is a totally geodesic K\"ahler submanifold and $N_\perp$ is a totally
umbilical totally real submanifold of $\tilde M$. Moreover, $N_T\times_f N_\perp$ is  minimal in $\tilde M$.

 When $M$ is anti-holomorphic, i.e., when $J\mathcal D^\perp_x=T^\perp_xN$, and $p>1$. The equality sign of \eqref{E:9.1}  holds identically if and only if   $N_\perp$ is a totally umbilical submanifold of $\tilde M$.

 If $M$ is anti-holomorphic and $p=1$, then the equality sign of \eqref{E:9.1} holds identically if and only if  the characteristic vector field $J\xi$ of $M$ is a principal vector field with zero as its principal curvature. (Notice that in this case, $M$ is a real hypersurface in $\tilde M$.) Also, in this case, the equality sign of \eqref{E:9.1} holds identically if and only if $M$ is a minimal hypersurface in $\tilde M$.
 \end{theorem}  

$CR$-warped products in complex space forms satisfying the equality case of \eqref{E:9.1} have been completely classified in \cite{c7,c8}.

\begin{theorem}\label{T:9.3} A $CR$-warped product $N_T\times_{f} N_\perp$ in  $\mathbb C^m$ satisfies
\begin{equation}||\sigma||^2= 2p||\nabla (\ln f)||^2\end{equation} identically if and only if the following four statements hold:
\begin{enumerate}
\item[(i)] $N_T$ is an open portion of a complex Euclidean $h$-space $\mathbb C^h$,

\item[(ii)] $N_\perp$ is an open portion of the unit $p$-sphere $S^p$, 

\item[(iii)] there exists  $a=(a_1,\ldots,a_h)\in S^{h-1}\subset {\mathbb E}^h$ such that $f=\sqrt{\left<a,z\right>^2 +\left<ia,z\right>^2}$ for  $z=(z_1,\ldots,z_h)\in{\mathbb C}^h,\, w=(w_0,\ldots,w_p)\in S^p\subset {\mathbb E}^{p+1}$, and

\item[(iv)] up to rigid motions, the immersion is given by
\begin{align}\notag \hbox{\bf x}(z,w)=&\Bigg(  z_1+(w_0-1)a_1\sum_{j=1}^h
a_jz_j,\,\cdots z_h+a(w_0-1)a_h\sum_{j=1}^ha_jz_j,\\& \hskip.5in w_1\sum_{j=1}^h a_jz_j,\,\,\ldots,w_p\sum_{j=1}^h a_jz_j,0,\,\ldots,0\Bigg).\notag\notag\end{align}
\end{enumerate}\end{theorem}

A $CR$-warped product $N_T\times_f N_\perp$ is said to be {\it trivial\/} if its warping function $f$ is constant. A trivial $CR$-warped product $N_T\times_f N_\perp$ is nothing but a $CR$-product $N_T\times N_\perp^f$, where $N_\perp^f$ is the  manifold  with metric $f^2g_{N_\perp}$ which is  homothetic to the original metric $g_{N_\perp}$ on $N_\perp$.

The following result from \cite{c8} completely classifies $CR$-warped products in complex projective spaces satisfying the equality case of \eqref{E:9.1} identically.

\begin{theorem}\label{T:9.4}  A non-trivial $CR$-warped product $\,N_T\times_{f} N_\perp$   in $CP^m$ satisfies the basic equality $||\sigma||^2= 2p||\nabla (\ln f)||^2$ if and only if we have
\begin{enumerate}
\item  $N_T$ is an open portion of complex Euclidean $h$-space ${\mathbb C}^h$,

\item $N_\perp$ is an open portion of a unit $p$-sphere $S^p$, and

\item up to rigid motions, the immersion {\bf x} of $N_T\times_{f} N_\perp$ into $CP^m$ is the composition $\pi\circ\breve{\hbox{\bf x}}$, where 
\begin{align}\notag  \breve{\hbox{\bf x}}&(z,w)=\Bigg( z_0+(w_0-1)a_0\sum_{j=0}^h a_j z_j,\,\cdots ,z_h+(w_0-1)a_h\sum_{j=0}^h a_jz_j,\\&\hskip.8inw_1\sum_{j=0}^h a_jz_j,\,\ldots,\;w_p\sum_{j=0}^h
a_jz_j,0,\,\ldots,0\Bigg),\notag\end{align} 
$\pi$ is the projection $\pi:{\mathbb C}^{m+1}_*\to CP^m$, $a_0,\ldots,a_h$ are real numbers satisfying $a_0^2+a_1^2+\cdots+a_h^2=1$, $z=(z_0,z_1,\ldots,z_h)\in {\mathbb C}^{h+1}$ and $w=(w_0,\ldots,w_p)\in S^p\subset {\mathbb E}^{p+1}$.
\end{enumerate} \end{theorem}

The following result from \cite{c8} completely classifies  $CR$-warped products in complex hyperbolic spaces satisfying the equality case of \eqref{E:9.1} identically.

\begin{theorem} \label{T:9.5}  A $CR$-warped product $N_T\times_{f} N_\perp$  in $CH^m$ satisfies the basic equality $$||\sigma||^2= 2p||\nabla (\ln f)||^2$$ if and only if one of the following two cases occurs:
\begin{enumerate}
\item $N_T$ is an open portion of complex Euclidean $h$-space ${\mathbb C}^h$, $N_\perp$ is an open portion of  a unit $p$-sphere $S^p$ and, up to rigid motions, the immersion is the composition $\pi\circ\breve{\hbox{\bf x}}$, where $\pi$ is the projection $\pi:{\mathbb C}^{m+1}_{*1}\to CH^m$ and 
\begin{align}\notag &\breve{\hbox{\bf x}}(z_,w)=\Bigg( z_0+a_0(1-w_0)\sum_{j=0}^h a_j z_j,z_1+a_1(w_0-1)\sum_{j=0}^h a_j z_j,\,\cdots, \\&\hskip.1 in  z_h+a_h(w_0-1)\sum_{j=0}^h a_jz_j,\;w_1\sum_{j=0}^h a_jz_j,\ldots,\; w_p\sum_{j=0}^h
a_jz_j,0,\ldots,0\Bigg)\notag \end{align} 
for some real numbers  $a_0,\ldots,a_{h}$  satisfying
$a_0^2-a_1^2-\cdots-a_{h}^2=-1$, where 
$z=(z_0,\ldots,z_h)\in{\mathbb C}^{h+1}_1$ and $
w=(w_0,\ldots,w_p)\in S^p\subset{\mathbb E}^{p+1}$.

\item $p=1$,   $N_T$ is an open portion of  ${\mathbb C}^h$ and, up to rigid motions, the immersion  is  the composition $\pi\circ\breve{\hbox{\bf x}}$, where 
\begin{align} \notag &\breve{\hbox{\bf x}}(z,t)=\Bigg(z_0+a_0(\cosh t-1)\sum_{j=0}^h a_j z_j,  z_1+a_1(1-\cosh t)\sum_{j=0}^h
a_jz_j,\\&\hskip.3in \ldots,z_h+a_h(1-\cosh t)\sum_{j=0}^h a_jz_j, \sinh t\sum_{j=0}^h a_jz_j,0,\ldots,0\Bigg)\notag
\end{align} 
for some real numbers $a_0,a_1\ldots,a_{h+1}$ satisfying $a_0^2-a_1^2-\cdots-a_{h}^2=1$.
\end{enumerate}\end{theorem}

A multiply warped product $N_T \times_{f_2} N_2\times \cdots \times_{f_k} N_k$ in a K\"ahler manifold $\tilde M$ is called a {\it multiply $CR$-warped product}  if $N_T$ is a holomorphic submanifold and $N_\perp={}_{f_2} N_2 \times \cdots\times_{f_k} N_k$ is a totally real submanifold of $\tilde M$.

Theorem \ref{T:9.2} was extended in \cite{CD08} to multiply $CR$-warped products in the following.

\begin{theorem}\label{T:2} Let $N=N_T \times_{f_2} N_2 \times \cdots \times_{f_k} N_k$ be a multiply $CR$-warped product in an arbitrary K\"ahler manifold $\tilde M$. Then the second fundamental form $\sigma$ and the warping functions $f_2,\ldots, f_k$ satisfy
 \begin{align} \label{0.6} &||\sigma||^2 \geq 2\sum_{i=2}^k n_i ||\nabla (\ln f_i)||^2. \end{align}

 The equality sign of \eqref{0.6} holds identically if and only if the following four statements hold:

\begin{enumerate}
\item[(a)] $N_T$ is a  totally geodesic holomorphic submanifold of $\tilde M$;

\item[(b)]  For each $i\in \{2,\ldots,k\}$,   $N_i$ is  a totally umbilical submanifold of $\tilde M$ with $-\nabla (\ln f_i)$ as its mean
curvature vector;

\item[(c)] ${}_{f_2} N_2 \times \cdots \times_{f_k} N_k$ is immersed as mixed totally geodesic submanifold in $\tilde M$;  and

 \item[(d)]  For each point $p\in N$, the first normal space ${\rm Im} \,h_p$ is a subspace of $J(T_pN_\perp)$.
\end{enumerate} \end{theorem}

\vskip.15in

\section{\uppercase{$CR$-warped products with compact holomorphic factor}}

When the holomorphic factor $N_{T}$ of a $CR$-warped product $N_{T}\times_{f}N_{\perp}$ is compact,  we have the following sharp results from \cite{c04-2}.

\begin{theorem} \label{T:10.1} Let $N_T\times_f N_\perp$ be a $CR$-warped product in the complex
projective $m$-space $CP^m$ of constant holomorphic sectional curvature $4$. If $N_T$ is compact, then $$m\geq h+p+hp.$$
\end{theorem}

\begin{theorem} \label{T:10.2} If $\,N_T\times_f N_\perp\,$ is a $CR$-warped product in $\,CP^{h+p+hp}\,$ with compact $N_T$, then $N_T$ is holomorphically isometric to $CP^h$.
\end{theorem} 

\begin{theorem} \label{T:10.3} For any $CR$-warped product $N_T\times_f N_\perp$ in $CP^m$ with compact $N_T$ and any $q\in N_\perp$, we have
\begin{align}\label{E:1.4}\int_{N_T\times\{q\}} ||\sigma||^2dV_{T}\geq 4hp\,\hbox{\rm vol}(N_T),\end{align} 
where $||\sigma||$ is the norm of the second fundamental form,  $dV_T$ is the volume element of $N_{T}$, and  $\hbox{\rm vol}(N_T)$ is the volume of $N_T$.

The equality sign of \eqref{E:1.4} holds
identically if and only if  we have:

\begin{enumerate}
\item The warping function $f$ is constant.

\item  $(N_T,g_{N_T})$ is holomorphically isometric to $CP^h$ and it is isometrically immersed in $CP^m$ as a totally geodesic complex submanifold.

\item $(N_\perp,f^2 g_{N_\perp})$ is isometric to an open portion of the real projective
$p$-space $RP^p$ of constant sectional curvature one  and it is isometrically immersed in
$CP^m$ as a totally geodesic totally real submanifold.

\item $N_T\times_f N_\perp$ is immersed linearly fully in a linear complex subspace $CP^{h+p+hp}$ of $CP^m$; and moreover, the immersion is rigid.
\end{enumerate}\end{theorem}

\begin{theorem} \label{T:10.4} Let $N_T\times_f N_\perp$ be a $CR$-warped product with compact $N_T$ in $CP^m$. If the warping function $f$ is non-constant, then, for each $q\in N_\perp$, we have
\begin{align}\label{E:1.3}\int_{N_T\times\{q\}} ||\sigma||^2dV_{T}\geq 2p\lambda_1 \int_{N_T}(\ln f)^2dV_T+ 4hp\,\hbox{\rm vol}(N_T),\end{align}  
where  $\lambda_1$  is the first positive eigenvalue of the Laplacian $\Delta$ of $N_T$.

Moreover, the equality sign of \eqref{E:1.3} holds identically if and only if  we have 

\begin{enumerate}
\item  $\Delta\ln f=\lambda_1 \ln f$.

\item The $CR$-warped product is both $N_T$-totally geodesic and $N_\perp$-totally geodesic.
\end{enumerate}\end{theorem}

The following example shows that Theorems \ref{T:10.3} and \ref{T:10.4}  are sharp.

\begin{example} {\rm Let $\iota_1$ be the identity map of $CP^h$ and let $$\iota_2:RP^p\to CP^p$$ be a totally geodesic Lagrangian embedding of $RP^p$ into $CP^p$. Denote by $$\iota=(\iota_1,\iota_2):CP^h\times
RP^p\to CP^h\times CP^p$$ the product embedding of $\iota_1$ and $\iota_2$. 
Moreover, let $S_{h,p}$ be the Segre embedding of  $CP^h\times CP^p$ into $CP^{hp+h+p}$. Then the composition $\phi=S_{h,p}\circ\iota$:
\begin{align}  CP^h &\times RP^p\xrightarrow[\text{totally geodesic}]  {\text{$(\iota_1,\iota_2)$}}  CP^h\times CP^p \xrightarrow[\text{Segre
embedding}] {\text{$S_{h,p}$}} CP^{hp+h+p}\notag\end{align} is a $CR$-warped product in $CP^{h+p+hp}$ whose holomorphic factor $N_T=CP^h$ is a compact manifold. 
 Since the second fundamental form of $\phi$ satisfies the equation: $||\sigma||^2=4hp$, we have the equality case of \eqref{E:1.4} identically. }
\end{example} 
 
The next example shows that the assumption of compactness in Theorems \ref{T:10.3} and \ref{T:10.4} cannot be removed.

\begin{example}{\rm  Let ${\mathbb C}^*={\mathbb C}-\{0\}$ and ${\mathbb C}_*^{m+1}={\mathbb C}^{m+1}-\{0\}$. Denote by $\{z_0,\ldots,z_h\}$   a natural complex coordinate system on ${\mathbb C}^{m+1}_*$.  

Consider the action of
${\mathbb C}^*$ on ${\mathbb C}_*^{m+1}$ given by
$$\lambda\cdot (z_0,\ldots,z_m)=(\lambda z_0,\ldots,\lambda z_m)$$ for $\lambda\in {\mathbb C}_*$.  Let
$\pi(z)$ denote the equivalent class containing $z$ under this action. Then the  set of equivalent classes is the complex projective $m$-space $CP^m$ with the  complex structure induced from the complex structure on ${\mathbb C}^{m+1}_*$. 

For any two natural numbers $h$ and $p$, we define a map: $$\breve \phi:\mathbb C^{h+1}_*\times S^p\to\mathbb C^{h+p+1}_*$$ by
$$\breve \phi(z_0,\ldots,z_h;w_0,\ldots,w_p)=\big(w_0z_0,w_1z_0,\ldots,w_pz_0,z_1,\ldots,z_h\big)$$ 
for $(z_0,\ldots,z_h)$ in $\mathbb C^{h+1}_*$ and $(w_0,\ldots,w_p)$ in $S^{p}$ with $\sum_{j=0}^p w_j^2=1$.

Since the image of $\breve \phi$ is invariant under the action of ${\mathbb C}_*$, the composition:
\begin{align}\pi\circ\breve \phi: {\mathbb C}^{h+1}_*\times S^p\xrightarrow[\text{}] {\text{$\breve \phi$}} {\mathbb C}^{h+p+1}_*\xrightarrow[\text{}] {\text{$\pi$}} CP^{h+p}\notag\end{align} 
induces a $CR$-immersion of the product manifold $N_T\times S^p$ into $CP^{h+p}$, where $$N_T=\big\{(z_0,\ldots,z_h)\in CP^h:z_0\ne 0\big\}$$ is a proper open subset of $CP^h$. Clearly, the induced metric on
$N_T\times S^p$ is a warped product metric and the holomorphic factor $N_T$ is non-compact. 

Notice that the complex dimension of the ambient space is  $h+p$; far less than $h+p+hp$.}
\end{example}

\vskip.2 in

\section{\uppercase{Another optimal inequality for $CR$-warped products}}  

All $CR$-warped products in complex space forms also satisfy another general optimal inequality obtained in \cite{c11}. 

\begin{theorem} \label{T:11.1} Let $N=N_T^h\times_f N_\perp^p$ be a  $CR$-warped product in a complex space form $\tilde M(4c)$ of constant holomorphic sectional curvature $c$. Then we have
\begin{equation}\label{E:11.1}||\sigma||^2\geq 2p\big\{||\nabla (\ln f)||^2+\Delta(\ln f)+2hc\big\}.\end{equation}
  
If the equality sign of \eqref{E:11.1}  holds identically, then $N_T$ is a totally geodesic submanifold and $N_\perp$ is a totally umbilical submanifold. Moreover, $N$ is a minimal submanifold in $\tilde M(4c)$. \end{theorem} 

The following three theorems  from \cite{c11} completely classify all $CR$-warped products which satisfy the equality case of \eqref{E:11.1} identically.

 \begin{theorem}\label{T:11.2}  Let
$\phi:N_T^h\times_f N_\perp^p\to {\mathbb C}^m$ be a  $CR$-warped product in  $\mathbb C^m$. Then we have
\begin{equation}\label{E:5.5}||\sigma||^2\geq 2p\big\{||\nabla(\ln f)||^2+\Delta(\ln f)\big\}.\end{equation}

 The  equality case of \eqref{E:5.5} holds identically if and only if the following four statements hold.
\begin{enumerate}
\item $N_T$ is an open portion of $\mathbb  C^h_*:=\mathbb  C^h-\{0\}$;

\item $N_\perp$ is an open portion of $S^p$;

\item There is $\alpha,\, 1\leq\alpha\leq h,$ and complex Euclidean  coordinates $\{z_1,\ldots,z_h\}$ on $\mathbb C^h$ such that $f=\sqrt{\sum_{j=1}^\alpha z_j\bar z_j}$\,; 

\item Up to rigid motions, the immersion $\phi$ is given by
\begin{equation}\notag\phi=\big(w_0 z_1,\ldots,w_p z_1,\ldots,w_0z_\alpha ,\ldots, w_p z_\alpha ,z_{\alpha
+1},\ldots,z_h,0,\ldots,0\big)\end{equation} for $z=(z_1,\ldots,z_h)\in {\mathbb C}_*^h$ and $\,w=(w_0,\ldots,w_p)\in S^p\subset{\mathbb E}^{p+1}$.
\end{enumerate} \end{theorem} 

 \begin{theorem} \label{T:11.3}   Let $\phi:N_T\times_{f} N_\perp\to CP^m$  be a $CR$-warped product  with $\dim_{\mathbb C} N_T=h$ and $\dim_{\mathbb R} N_\perp=p$. Then we have
\begin{equation}\label{E:9.9}||\sigma||^2\geq 2p\big\{||\nabla (\ln f)||^2+\Delta(\ln f)+2h\big\}.\end{equation}

 The $CR$-warped product  satisfies the equality case of \eqref{E:9.9} identically if and only if the following three statements hold.

\begin{enumerate}
\item[(a)]  $N_T$ is an open portion of  complex projective $h$-space $CP^h$;

\item[(b)]   $N_\perp$ is an open portion of  unit $p$-sphere $S^p$; and

\item[(c)]  There exists a natural number $\alpha\leq h$ such that, up to rigid motions,  $\phi$ is the composition $\pi\circ\breve{\phi}$, where 
\begin{equation}\notag \breve{\phi}(z,w)=\big(w_0 z_0,\ldots,w_p
z_0,\ldots,w_0z_\alpha ,\ldots, w_p z_\alpha ,z_{\alpha +1},\ldots,z_h,0\ldots,0\big)\end{equation} for $z=(z_0,\ldots,z_h)\in
{\mathbb C}_*^{h+1}$ and $\,w=(w_0,\ldots,w_p)\in
S^p\subset{\mathbb E}^{p+1}$, where  $\pi$ is the projection $\pi:{\mathbb C}^{m+1}_*\to CP^m$.
\end{enumerate} \end{theorem}

 \begin{theorem} \label{T:11.4}  Let $\phi:N_T\times_{f} N_\perp\to CH^m$  be a $CR$-warped product  with
$\dim_{{\mathbb C}} N_T=h$ and $\dim_{{\mathbb R}} N_\perp=p$. Then we have
\begin{equation}\label{E:9.11}||\sigma||^2\geq 2p\big\{||\nabla (\ln f)||^2+\Delta(\ln f)-2h\big\}.\end{equation}

 The $CR$-warped product satisfies the equality case of \eqref{E:9.11} identically if and only if the following three statements hold.

\begin{enumerate}
\item[(a)]   $N_T$ is an open portion of  complex hyperbolic $h$-space $CH^h$;

\item[(b)]  $N_\perp$ is an open portion of  unit $p$-sphere $S^p$ $($or {$\mathbb R$}, when $p=1$$)$; and

\item[(c)]  up to rigid motions, $\phi$ is the composition $\pi\circ\breve{\phi}$, where either $\breve\phi$ is  given by
\begin{equation}\notag\breve{\phi}(z,w)=\big(z_{0},\ldots,z_{\beta},w_0 z_{\beta+1},\ldots,w_p z_{\beta+1},\ldots,w_0z_h ,\ldots, w_p z_h ,0\ldots,0\big)\end{equation} for  $0<\beta\leq h$,  $z=(z_0,\ldots,z_h)\in {\mathbb C}_{*1}^{h+1}$ and $\,w=(w_0,\ldots,w_p)\in S^p$, or $\breve \phi$ is given by
\begin{align} \notag&\breve{\phi}(z,u)=\big(z_0\cosh u,z_0\sinh u,z_1\cos
u,z_1\sin u,\ldots,\ldots,z_\alpha\cos u,z_\alpha\sin u,\\&\hskip.9inz_{\alpha+1},\ldots,z_h,0\ldots,0\big) \notag\end{align}
 for $z=(z_0,\ldots,z_h)\in {\mathbb C}_{*1}^{h+1}$, where  $\pi$ is the projection $\pi:{\mathbb C}^{m+1}_{*1}\to CH^m(-4)$.
\end{enumerate} \end{theorem}

\vskip.2 in

\section{\uppercase{Warped Product $CR$-submanifolds and $\delta$-invariants}}

Let $N$ be a $CR$-submanifold of a K\"ahler manifold with holomorphic distribution $\mathcal D$ and totally real distribution $\mathcal D^\perp$. 
We define the {\it $CR$ $\delta$-invariant} $\delta(\mathcal D)$ of $N$ by
\begin{align}\label{4.1} \delta(\mathcal D)(x)=\tau(x)-\tau(\mathcal D_x),\end{align}
where $\tau$ and $\tau(\mathcal D)$ are the scalar curvature of $N$ and the scalar curvature of the holomorphic distribution $\mathcal D\subset TN$, respectively (see \cite{c11,c12} for details).

The following result from \cite{c12} provides a general optimal inequality involving the $CR$ $\delta$-invariant for $CR$-warped submanifolds in complex space forms.

\begin{theorem} \label{T:12.1} Let $N=N^T\times_f N^\perp$ be a $CR$-warped product in a complex space form $\tilde M^{h+p}(4c)$ with $h=\dim_{C}N^T\geq 1$ and $p=\dim N^\perp\geq 2$. Then 
\begin{align}\label{12.1} H^2\geq \frac{2(p+2)}{(2h+p)^2(p-1)}\! \left\{\delta(\mathcal D)-\frac{p\Delta f}{f}-\frac{p(p-1)c}{2}\right\},\end{align}
where $\Delta f$ is the Laplacian of the warping function $f$ and $H^{2}$ is the squared mean curvature.

The equality sign of \eqref{12.1} holds at a point $x\in N$ if and only if there exists an orthonormal basis $\{e_{2h+1},\ldots,e_n\}$ of $\mathcal D_x^\perp$ such that the coefficients of the second fundamental $\sigma$ with respect to $\{e_{2h+1},\ldots,e_n\}$ satisfy
\begin{equation}\begin{aligned}\notag &\sigma^r_{rr}=3\sigma^r_{ss}, \hskip.2in {\rm for}\;\;  2h+1\leq r\ne s\leq 2h+p,\\& \sigma^r_{st}=0,\hskip .4in {\rm for \; distinct}\; \; r,s,t\in \{2h+1,\ldots,2h+p\}.\end{aligned}\end{equation}
\end{theorem}

All $CR$-warped products in ${\mathbb C}^{h+p}$ satisfying the equality case of \eqref{12.1} identically are completely classified in \cite{c12} as follows.

\begin{theorem} \label{T:12.2} Let $\psi:N^T\times_f N^\perp\to {\mathbb C}^{h+p}$ be a $CR$-warped product in ${\mathbb C}^{h+p}$ with $h=\dim_{C}N^T\geq 1$ and $p=\dim N^\perp\geq 2$. Then 
\begin{align}\label{12.2} H^2\geq \frac{2(p+2)}{(2h+p)^2(p-1)}\! \left\{\delta(\mathcal D)-\frac{p\Delta f}{f}\right\}.\end{align}

The equality sign of \eqref{12.2} holds identically if and only if, up to dilations and rigid motions of ${\mathbb C}^{h+p}$, one of the following three cases occurs: 
\vskip.04in 
\begin{enumerate}
\item[{\rm (a)}] The $CR$-warped product is an open part of the CR-product ${\mathbb C}^h\times W^p\subset {\mathbb C}^h\times {\mathbb C}^p$, where $W^{p}$ is the Whitney $p$-sphere in ${\mathbb C}^{p}$;
\vskip.04in 

\item[{\rm (b)}] $N^{T}$ is an open part of ${\mathbb C}^h$,  $N^{\perp}$ is an open part of
the unit $p$-sphere $S^p$, $f=|z_{1}|$ and $\psi$ is the minimal immersion defined by
\begin{equation}\begin{aligned}\notag &\big(  z_1w_0,\,\cdots, z_1w_{p},z_{2},\ldots,z_{h}\big),\end{aligned}\end{equation}
 where $z=(z_1,\ldots,z_h)\in {\mathbb C}^h$ and $ w=(w_0,\ldots,w_p)\in
S^p\subset {\mathbb E}^{p+1}$;
 
 \item[{\rm (c)}] $N^{T}$ is an open part of ${\mathbb C}^h$,  $N^{\perp}$ is the warped product of a curve and an open part of $S^{p-1}$ with warping function $\varphi=\frac{\sqrt{c^{2}-1}}{\sqrt{2}}{\rm cn}\big(c\,t, \tfrac{\sqrt{c^{2}-1}}{\sqrt{2}c}\big), c>1$,  $f=|z_{1}|$, and $\psi$ is the non-minimal immersion defined by
\begin{equation}\begin{aligned}\notag & \hskip.3in
\left(\!z_{1}  e^{ \int \frac{\varphi(\varphi' +{\rm i}k \varphi^{2})}{\varphi^{2}-1}dt},
 z_{1} \varphi e^{{\rm i}k\! \int \! \varphi dt}
w_{1}, \cdots
 z_{1} \varphi e^{{\rm i}k\! \int \! \varphi dt}w_{p},z_{2},\ldots,z_{h}\!\right)\! ,\end{aligned}\end{equation}
with  $z=(z_1,\ldots,z_h)\in {\mathbb C}^h$,  $(w_{1},\ldots,w_{p})\in S^{{p-1}}(1)\subset {\mathbb E}^{p}$, and  $k={\sqrt{c^{4}-1}}/{2}$.
 
\end{enumerate}\end{theorem}

\vskip.1in

\section{\uppercase{Warped product real hypersurfaces in complex space forms}}

We have  the following non-existence theorem from \cite{CM} for Riemannian product real hypersurfaces in complex space forms, .

\begin{theorem}\label{T:13.1} There do not exist real hypersurfaces in complex projective and complex hyperbolic spaces which are Riemannian products of two or more Riemannian manifolds of positive dimension. 

In other words, every real hypersurface in a nonflat complex space form is irreducible.
\end{theorem}

A contact manifold is an odd-dimensional manifold $M^{2n+1}$ with  a 1-form $\eta$ such that $\eta\wedge(d\eta)^n\not=0$.  A curve $\gamma=\gamma(t)$ in a contact manifold is called a {\it Legendre curve} if $\eta(\beta'(t))=0$ along $\beta$. Let $S^{2n+1}(c)$ denote the hypersphere in ${\mathbb C}^{n+1}$ with curvature  $c$ centered at the origin. Then $S^{2n+1}(c)$ is a contact manifold endowed with a canonical contact structure which is the dual 1-form of the characteristic vector field $J\xi$, where $J$ is the complex structure and $\xi$ the unit normal vector on $S^{2n+1}(c)$.

 Legendre curves are known to play an important role in the study of contact manifolds, e.g. a diffeomorphism of a contact manifold is a contact transformation if and only if it maps Legendre curves to Legendre curves.

Contrast to Theorem \ref{T:13.1}, there exist many warped product real hypersurfaces in complex space forms as given in the following three theorems from \cite{c02-3}.

\begin{theorem}\label{T:13.2}  Let $a$ be a positive number and $\gamma(t)=(\Gamma_1(t),\Gamma_2(t))$ be a unit speed Legendre curve $\gamma:I\to S^3(a^2)\subset{\mathbb C}^2$ defined on an open interval  $I$.  Then
\begin{align} &\hbox{\bf x}(z_1,\ldots,z_{n},t)=\big(a\Gamma_1(t)z_1, a\Gamma_2(t) z_1,z_2,\ldots,z_n\big),\quad z_1\ne 0\end{align} 
defines a real hypersurface which is the warped product ${\mathbb C}_{**}^{n}\times_{a|z_1|} I$ of a complex $n$-plane and $I$, where ${\mathbb C}_{**}^{n}=\{(z_1,\ldots,z_n):z_1\ne 0\}$.

Conversely, up to rigid motions of ${\mathbb C}^{n+1}$, every  real hypersurface in ${\mathbb C}^{n+1}$ which is the warped product $N\times_f I$ of a complex hypersurface $N$  and an open interval $I$ is either obtained in the way described above or given by the product submanifold ${\mathbb C}^n\times C\subset {\mathbb C}^{n}\times{\mathbb C}^1$ of ${\mathbb C}^{n}$ and a real curve $C$ in ${\mathbb C}^1$.
\end{theorem}

 Let $S^{2n+3}$ denote the unit hypersphere in  ${\mathbb C}^{n+2}$ centered at the origin and put $$U(1)=\{\lambda\in {\mathbb C}:\lambda\bar\lambda=1\}.$$ Then there is a $U(1)$-action on $S^{2n+3}$ defined by $z\mapsto \lambda z$. At $z\in S^{2n+3}$ the vector $V=iz$ is tangent to the flow of the action. The quotient space $S^{2n+3}/\sim$, under the identification induced from the action, is a complex projective space $ CP^{n+1}$ which endows with the canonical Fubini-Study metric of constant holomorphic sectional curvature $4$. 
 
 The almost complex structure $J$ on $ CP^{n+1}$ is induced from the complex structure $J$ on ${\mathbb C}^{n+2}$ via the Hopf fibration: $\,\pi : S^{2n+3}\to CP^{n+1}$. It is well-known that the Hopf fibration $\pi$ is a Riemannian submersion such that $V=iz$ spans the vertical subspaces. 

Let $\phi:M\to CP^{n+1}$ be an isometric immersion. Then $\hat M=\pi^{-1}(M)$ is a principal circle bundle over $M$ with totally geodesic fibers. The lift $\hat \phi:\hat M\to S^{2n+3}$ of $\phi$ is an isometric immersion so that the diagram:
\[\label{E:5.b} \begin{CD} \hat M @>\hat \phi >>S^{2n+3}\\ @V{\pi}VV @VV{\pi}V\\ M @> \phi>>  CP^{n+1}\end{CD} \]
commutes.
 
Conversely, if $\psi:\hat M \to S^{2n+3}$ is an isometric immersion which is invariant under the $U(1)$-action, then there is a unique isometric immersion $\psi_\pi:\pi(\hat M)\to CP^{n+1}$ such that the associated diagram commutes. We simply call the immersion $\psi_\pi:\pi(\hat M)\to CP^{n+1}$ {\it the projection} of $\psi:\hat M\to S^{2n+3}$.

For a given  vector $X\in T_z(CP^{n+1})$ and a point $u\in S^{2n+2}$ with $\pi(u)=z$, we denote by $X_u^*$ the horizontal lift of $X$ at $u$ via $\pi.$ There exists  a canonical orthogonal decomposition:
\begin{equation}T_u S^{2n+3}=(T_{\pi(u)}CP^{n+1})_u^*\oplus
\hbox{\rm Span}\,\{V_u\}.\end{equation}
Since $\pi$ is a Riemannian  submersion, $X$ and  $X^*_u$ have the same length.

We put
\begin{equation}S_*^{2n+1}=\left\{(z_0,\ldots,z_n):\sum_{k=0}^n
z_k\bar z_k=1,\, z_0\ne 0\right\},\quad CP^n_0=\pi(S_*^{2n+1}).\end{equation}

The following theorem from  \cite{c02-3} classifies all warped products hypersurfaces of the form $N\times_{f}I$ in complex projective spaces.

\begin{theorem}\label{T:13.3}   Suppose that $a$ is a
positive number and $\gamma(t)=(\Gamma_1(t),\Gamma_2(t))$ is a
unit speed Legendre curve $\gamma:I\to S^3(a^2)\subset{\mathbb C}^2$ defined on an open interval $I$.  Let $\hbox{\bf x}:S_*^{2n+1}\times I\to{\mathbb C}^{n+2}$ be the map defined by
\begin{align} &\hbox{\bf x}(z_0,\ldots,z_{n},t)=\big(a\Gamma_1(t)z_0,a\Gamma_2(t) z_0,z_1,\ldots,z_{n}\big),\;\;\;\sum_{k=0}^{n} z_k\bar z_k=1.\end{align} 
Then 
\begin{enumerate}
\item $\,\hbox{\bf x}$ induces an isometric immersion $\psi:S_*^{2n+1}\times_{a|z_0|} I\to S^{2n+3}$.

\item  The image $\psi(S_*^{2n+1}\times_{a|z_0|} I)$ in $S^{2n+3}$ is invariant under the action of $U(1)$.

\item the projection $\psi_\pi:\pi(S_*^{2n+1}\times_{a|z_0|} I) \to CP^{n+1}$ of $\psi$ via $\pi$ is a warped product hypersurface $CP^{n}_0\times_{a|z_0|} I$ in $CP^{n+1}$.
\end{enumerate}

Conversely, if a  real  hypersurface in $CP^{n+1}$ is a warped product $N\times_f I$ of a complex hypersurface $N$ of $CP^{n+1}$  and an open interval $I$, then, up to rigid motions, it is locally obtained in the way described above.
\end{theorem}

In the complex pseudo-Euclidean space ${\mathbb C}^{n+2}_1$ endowed with  pseudo-Euclidean metric
\begin{equation}g_0=-dz_0d\bar z_0 +\sum_{j=1}^{n+1}dz_jd\bar z_j,\end{equation} we  define the anti-de Sitter space-time  by
\begin{equation}H^{2n+3}_1=\big\{(z_0,z_1,\ldots,z_{n+1}):
\left<z,z\right>=-1\big\}.\end{equation} It is known that $H^{2n+3}_1$ has constant sectional curvature $-1$. There is a $U(1)$-action on
$H_1^{2n+3}$ defined by $z\mapsto \lambda z$. At a point $z\in H^{2n+3}_1$,  $iz$ is tangent to the flow of the action. The orbit is given by  $z_t=e^{it}z$ with ${{dz_t}\over{dt}}=iz_t$ which lies in the negative-definite plane spanned by $z$ and $iz$. 

The quotient space $H^{2n+3}_1/\sim$ is the complex hyperbolic space $CH^{n+1}$ which endows a canonical K\"ahler metric of constant holomorphic sectional curvature $-4$. The  complex structure $J$ on $ CH^{n+1}$ is induced from the canonical complex structure $J$ on $ {\mathbb C}^{n+2}_1$ via the totally geodesic fibration: $\pi:H^{2n+3}_1\rightarrow  CH^{n+1}.$

Let $\phi:M\to CH^{n+1}$ be an isometric immersion. Then $\hat M=\pi^{-1}(M)$ is a principal circle bundle over $M$ with totally geodesic fibers. The lift $\hat \phi:\hat M\to H_1^{2n+3}$ of $\phi$ is an isometric immersion such that the
 diagram:
\[\begin{CD} \hat M @>\hat \phi >>H_1^{2n+3}\\ @V{\pi}VV @VV{\pi}V\\ M @> \phi>>  CH^{n+1}\end{CD}\]
commutes.

Conversely, if $\psi:\hat M \to H_1^{2n+3}$ is an isometric immersion which is invariant under the $U(1)$-action, there
is a unique isometric immersion $\psi_\pi:\pi(\hat M)\to CH^{n+1}$,  called the {\it projection of\/} $\psi$ so that the associated diagram commutes.

We put
\begin{align} & H_{1*}^{2n+1}=\big\{(z_0,\ldots,z_{n})\in H^{2n+1}_1: z_n\ne 0\big\},\\& CH^n_*=\pi(H^{2n+1}_{1*}).\end{align}

The next theorem from  \cite{c02-3} classifies all warped products hypersurfaces of the form $N\times_{f}I$ in complex hyperbolic spaces.

 \begin{theorem}\label{T:13.4}  Suppose that $a$ is a positive number and $\gamma(t)=(\Gamma_1(t),\Gamma_2(t))$ is a unit speed Legendre curve $\gamma:I\to S^3(a^2)\subset{\mathbb C}^2$.  Let $\hbox{\bf y}:H_{1*}^{2n+1}\times I\to{\mathbb C}^{n+2}_1$ be the map defined by
\begin{align} &\hbox{\bf y}(z_0,\ldots,z_{n},t)=(z_0,\ldots,z_{n-1},
a\Gamma_1(t)z_{n},a\Gamma_2(t) z_{n}),\\&\hskip 1in z_0\bar z_0-\sum_{k=1}^{n} z_k\bar z_k=1.\end{align}
Then we have
\begin{enumerate}
\item  $\,\hbox{\bf y}$  induces an isometric immersion $\psi:H_{1*}^{2n+1}\times_{a|z_n|} I\to H_1^{2n+3}$.

\item  The image $\psi(H_{1*}^{2n+1}\times_{a|z_n|} I)$ in $H_1^{2n+3}$ is invariant under the $U(1)$-action.

\item the projection $\psi_\pi:\pi(H_{1*}^{2n+1}\times_{a|z_n|} I) \to CH^{n+1}$ of $\psi$ via $\pi$ is a warped product hypersurface $CH^{n}_*\times_{a|z_n|} I$ in $CH^{n+1}$.
\end{enumerate}

Conversely, if a  real hypersurface in $CH^{n+1}$ is a warped product $N\times_f I$ of a complex hypersurface $N$ and an open interval $I$, then, up to rigid motions, it is locally obtained in the way described above.
\end{theorem}

\vskip.1in

\section{\uppercase{Warped Product $CR$-submanifolds of nearly K\"ahler manifolds}}

Let $M$ be an almost Hermitian manifold with metric tensor $g$ and almost complex structure $J$.  Then, according to A. Gray \cite{G} in 1970, $M$ is called a {\it nearly K\"ahler manifold} provided that   \begin{align}\label{14.1}(\nabla_XJ )X=0, \;\; \forall X\in TM.\end{align}  Historically speaking, nearly K\"ahler manifolds are exactly the {\it Tachibana manifolds} initially studied in  S. Tachibana \cite{T} around 1959. 

Nearly K\"ahler manifolds form an interesting class of manifolds admitting a metric connection with parallel totally antisymmetric torsion (see \cite{Ag}).

The best known example of nearly K\"ahler manifolds, but not K\"ahlerian, is $S^6(1)$ with the nearly K\"ahlerian structure  induced from the vector cross product on the space of purely imaginary Cayley numbers $\mathcal O$. More general examples of nearly K\"ahler manifolds are the homogeneous spaces $G/K$, where $G$ is a compact semisimple Lie group and $K$ is the fixed point set of an automorphism of $G$ of order 3 \cite{WG68}.

 Strict nearly K\"ahler manifolds obtained a lot of consideration in the 1980s due to their relation to Killing spinors.  
Th. Friedrich and R. Grunewald showed in \cite{FG} that a 6-dimensional Riemannian manifold admits a Riemannian Killing spinor if and only if it is nearly K\"ahler. 

The only known 6-dimensional strict nearly K\"ahler manifolds are
$$S^{6}=G_{2}/SU(3). \; Sp(2)/SU(2)\times U(1),\; SU(3)/U(1)\times U(1),\; S^{3}\times S^{3}.$$ In fact, these are the only homogeneous nearly K\"ahler manifolds in dimension six \cite{Bu}. P.-A.  Nagy proved in  \cite{Nagy} that indeed any strict and complete nearly K\"ahler manifold is locally a Riemannian product of homogeneous nearly K\"ahler spaces, twistor spaces over K\"ahler manifolds and 6-dimensional nearly K\"ahler manifolds. 

The non-existence result for warped products $N_{\perp}\times_{f}N_{T}$ in K\"ahler manifolds, Theorem \ref{T:9.1}, was extended in \cite{KKS07,SG08} to warped products in nearly K\"ahler manifolds in the following.

 \begin{theorem} \label{T:14.1}There does not exist a proper warped product CR-submanifold of the form $N_{\perp}\times_{f}N_{T}$ with $N_{\perp}$ a totally real submanifold and $N_{T}$ a holomorphic submanifold, in a nearly K\"ahler manifold.\end{theorem}
 
 This theorem was further extended in \cite{KK09} to warped products $N\times_{f}N_{T}$  in a nearly K\"ahler manifold with $N$ and $N_{T}$ being Riemannian and holomorphic submanifolds of $\tilde M$, respectively.

 Similarly, Theorem \ref{T:9.3}  was also extended in \cite{AAK09,KKS07,SG08} to nearly K\"ahler manifold as follows.

 \begin{theorem}\label{T:14.2} Let $M=N_{T}\times_{f}N_{\perp}$ be a $CR$-warped submanifold of a nearly K\"ahler manifold $\bar M$. Then we have
 
\begin{enumerate}
 \item The squared norm of the second fundamental form $\sigma$ satisfies
 \begin{equation}\label{1} ||\sigma||^{2}\geq 2p ||\nabla \ln f||^{2},\end{equation}
 where $p$ is the dimension of $N_{\perp}$.
 
 \item If the equality in \eqref{1} holds identically, then $N_{T}$ is a totally geodesic submanifold, $N_{\perp}$ a totally umbilical submanifold, and $M$ is a minimal submanifold  of $\bar M$.
 \end{enumerate} \end{theorem}

It was shown in \cite{KJA} that there does not exist a proper doubly warped product submanifold of a nearly K\"ahler manifold $\tilde M$ with one of the factors a holomorphic submanifold. It also shows that there do not exist doubly twisted product generic submanifolds of nearly K\"ahler manifolds in the form ${}_{f_{2}}N_{T}\times_{f_{1}}N_{1}$ such that $N_{T}$ is holomorphic and $N_{1}$ is an arbitrary submanifold in $\tilde M$.

\vskip.1in

\section{\uppercase{Warped product submanifolds in para-K\"ahler manifolds}}

\def\p{{\mathcal{P}}}
\def\D{{\mathcal{D}}}
\def\Dp{{\mathcal{D}}^\perp}
\def\Ni{{N_\top}}
\def\Na{N_\perp}
\def\j{{\rm j}}

An almost para-Hermitian manifold is a manifold $\widetilde M$ equipped with an almost product structure
$\p\ne \pm I$ and a pseudo-Riemannian metric $\widetilde g$ such that
\begin{align}\label{15.1}\p^2=I,\;\; \widetilde g(\p X,\p Y)=-\widetilde g(X,Y),\end{align}
for vector fields $X$, $Y$ tangent to $\widetilde M$, where $I$ is the identity map. Clearly, it follows from \eqref{15.1} that the dimension of $\widetilde M$ is even and the metric $\widetilde g$ is neutral. An almost para-Hermitian manifold  is called  {\it para-K\"ahler}  if it satisfies $\widetilde \nabla \p=0$ identically, where  $\widetilde \nabla$ is the Levi Civita connection of $\widetilde M$. We define $||X ||_{2}$ associated with $\widetilde g$ on $\widetilde M$  by $||X||_2=\widetilde g(X,X)$.

 A pseudo-Riemannian submanifold $M$ of a para-K\"ahler manifold $\widetilde M$ is called {\it invariant} if the tangent bundle of $M$ is invariant under the action of $\p$. $M$ is called {\it anti-invariant}  if $\p$ maps each tangent space $T_pM, \, p\in M,$ into the normal space $T_p^\perp M$ (cf. \cite{c11}).

 A pseudo-Riemannian submanifold $M$ of a para-K\"ahler manifold $\widetilde M$ is called a
\emph{$\p R$-submanifold} if the tangent bundle $TM$ of $M$ is the direct sum of  
an \emph{invariant} distribution $\D$ and an
\emph{anti-invariant} distribution $\Dp$, i.e.,  $$T(M)=\D\oplus\Dp, \;\; \p\D=\D,\;\;  \p\Dp\subseteq T_p^\perp(M).$$
 
A $\p R$-submanifold is called a \emph{$\p R$-warped product} if it is a warped product $\Ni \times_f \Na$ of an invariant submanifold $\Ni$ and an anti-invariant submanifold $\Na$. 

The notion of $\p R$-warped product submanifolds in para-K\"ahler manifolds are introduced and studied in \cite{CMu12}. In particular, the following results were obtained in \cite{CMu12}.

\begin{theorem}\label{T:15.1}
\label{small_codim}
Let $\Ni\times_{f}\Na$ be a $\p R$-warped product in a para-K\"ahler manifold $\widetilde M$.
If $\Na$ is space-like $($respectively, time-like\/$)$ and $\dim \widetilde M= \dim \Ni+2\dim \Na$,
then the second fundamental form $\sigma$ of $\Ni\times_{f}\Na$ satisfies
\begin{equation}
\label{ineq_small:eq}
||\sigma||^{2} \leq 2p {||\nabla\ln f||}_2\;\; \hbox{ {\rm (respectively,} $||\sigma||^{2} \geq 2p {||\nabla\ln f||}_2)$}.
\end{equation}

If the equality sign of \eqref{ineq_small:eq} holds identically,  we have
\begin{equation}
\label{eq} \sigma(\mathcal D,\mathcal D)= \sigma(\mathcal D^\perp,\Dp)=\{0\}.\end{equation}
\end{theorem}

Para-complex numbers were introduced by J. T. Graves in 1845 as a generalization of complex numbers. Such numbers have the expression $v=x+\j y$, where $x,y$ are real numbers and $\j$ satisfies $\j^{2}=1,\,\j\ne 1$. 
The conjugate of $v$ is $\bar v=x-\j y$. The multiplication of two para-complex numbers is defined by
$$(a+\j b)(s+\j t)=(as+bt)+\j(at+bs).$$

For each natural number $m$, we put $$\mathbb D^{m}=\{(x_1+\j y_1,\ldots,x_m+\j y_m) : x_i, y_i\in{\mathbb{R}}\}.$$ With respect to the  multiplication of  para-complex numbers and the canonical flat metric,  $\mathbb D^{m}$ is a flat para-K\"ahler manifold of dimension $2m$. 
Once we  identify $$(x_1+\j y_1,\ldots,x_m+\j y_m)\in \mathbb D^{m}$$ with $(x_1,\ldots,x_m,y_1,\ldots,y_m)\in {\mathbb{E}}^{2m}_m$, we may identify $\mathbb D^{m}$ with the {\it para-K\"ahler $m$-plane} $\p^{m}$ in a natural way.

In the following we denote by $\mathbb S^{p}, \mathbb E^{p}$ and  $\mathbb  H^{p}$ the unit $p$-sphere, the Euclidean $p$-space and the unit hyperbolic $p$-space, respectively.

The following theorem from \cite{CMu12} completely classifies space-like $\p R$-warped products  in $\p^{m}$  which satisfy the equality case of \eqref{ineq_small:eq} identically.

\begin{theorem} \label{T:15.2}
Let $\Ni\times_f\Na$ be a space-like $\p R$-warped product  in the  para-K\"ahler $(h+p)$-plane $\p^{h+p}$ with $h=\frac{1}{2}\dim \Ni$ and $p=\dim \Na$. Then we have 
\begin{align}\label{IN} ||\sigma||^{2}\leq 2p{||\nabla\ln f||}_2.\end{align} 
The equality sign of \eqref{IN} holds identically if and only if  $\Ni$ is an open part of a para-K\"ahler $h$-plane, $\Na$ is an open part of $\mathbb S^{p},\, \mathbb E^{p}$ or $\mathbb H^{p}$,  and the immersion  is given by one of the following:

\begin{enumerate}
\item $\Phi:D_{1}\times_f {\mathbb{S}}^p\longrightarrow {\p}^{h+p}$;
\begin{equation}
\label{case1}
\begin{array}{c}
\Phi(z,w)=\text{\small$\Bigg($}z_1+\bar v_1(w_0-1)\sum\limits_{j=1}^hv_jz_j,\ldots,z_h+\bar v_h(w_0-1)\sum\limits_{j=1}^hv_jz_j,\\ \qquad w_1\sum\limits_{j=1}^h{\rm j} v_jz_j,\ldots,w_p\sum\limits_{j=1}^h\j v_jz_j \text{\small$\Bigg)$},\;\; h\geq 2,
\end{array}
\end{equation}
with warping function $f=\sqrt{\langle \bar v,z\rangle^2-\langle \j \bar v,z\rangle^2},$
where $v=(v_1,\ldots,v_h)\in{\mathbb{S}}^{2h-1}\subset {\mathbb{D}}^h$, $ w=(w_0,w_1,\ldots,w_p)\in{\mathbb{S}}^p$, $ z=(z_1,\ldots,z_h)\in D_{1}$ 
and $D_{1}=\left\{z\in{\mathbb{D}}^h : \langle \bar v,z\rangle^2>\langle \j \bar v,z\rangle^2\right\}$.

\item $\Phi:D_{1}\times_f {\mathbb{H}}^p\longrightarrow {\p}^{h+p}$;
\begin{equation}
\label{case2}
\begin{array}{c}
\Phi(z,w)=\text{\small$\Bigg($}z_1+\bar v_1(w_0-1)\sum\limits_{j=1}^hv_jz_j,\ldots,z_h+\bar v_h(w_0-1)\sum\limits_{j=1}^hv_jz_j,\\
\qquad w_1\sum\limits_{j=1}^h\j v_jz_j,\ldots,w_p\sum\limits_{j=1}^h\j v_jz_j \text{\small$\Bigg)$},\;\; h\geq 1,
\end{array}
\end{equation}
with the warping function $f=\sqrt{\langle \bar v,z\rangle^2-\langle \j \bar v,z\rangle^2}$, where $v=(v_1,\ldots,v_h)\in{\mathbb{H}}^{2h-1}\subset {\mathbb{D}}^h$,
$w=(w_0,w_1,\ldots,w_p)\in{\mathbb{H}}^p$  and 
$z=(z_1,\ldots,z_h)\in D_{1}$.

\item $\Phi(z,u):D_{1}\times_f {\mathbb{E}}^p\longrightarrow {\p}^{h+p}$;
\begin{equation}
\label{case3}
\begin{array}{c}
\Phi(z,u)=\text{\small$ \Bigg($}z_1+\dfrac{\bar v_1}{2}\Big(\sum\limits_{a=1}^pu_a^2\Big)\sum\limits_{j=1}^hv_jz_j,\ldots,
        z_h+\dfrac{\bar v_h}{2}\Big(\sum\limits_{a=1}^pu_a^2\Big)\sum\limits_{j=1}^hv_jz_j,\\
        u_1\sum\limits_{j=1}^h\j v_jz_j,\ldots,u_p\sum\limits_{j=1}^h\j v_jz_j \text{\small$ \Bigg)$},\;\; h\geq 2,
\end{array}\end{equation}
with the warping function $f=\sqrt{\langle \bar v,z\rangle^2-\langle \j \bar v,z\rangle^2}$, where $v=(v_1,\ldots,v_h)$ is a light-like vector in ${\mathbb{D}}^h$,
$z=(z_1,\ldots,z_h)\in D_{1}$ and $u=(u_1,\ldots,u_p)\in{\mathbb{E}}^p,$

Moreover, in this case, each leaf $\, {\mathbb{E}}^p$ is quasi-minimal in ${\p}^{h+p}$.

\item $\Phi(z,u):D_{2}\times_f {\mathbb{E}}^p\longrightarrow {\p}^{h+p}$;
\begin{equation}
\label{case4}
\begin{array}{c}  \Phi(z,u)=\text{\small$\Bigg( $}z_1+\dfrac{v_1}{2}\! \sum\limits_{a=1}^pu_a^2,\ldots,
        z_h+\dfrac{v_h}{2}\!\sum\limits_{a=1}^pu_a^2,
        \dfrac{v_0}{2}u_1,\ldots,\dfrac{v_0}{2}u_p \text{\small$\Bigg)$},\; h\geq 1, \end{array}
\end{equation}
with the warping function $f=\sqrt{-\langle v,z\rangle}$, where $v_0=\sqrt{b_1}+\epsilon\j\sqrt{b_1}$ with $b_1>0$,  $D_{2}=\{z\in{\mathbb{D}}^h:\langle v,z\rangle<0\}$,
$v=(v_1,\ldots,v_h)=(b_1+\epsilon\j b_1,\ldots,b_h+\epsilon\j b_h)$, $\epsilon=\pm1$,
$z=(z_1,\ldots,z_h)\in D_2$ and $ u=(u_1,\ldots,u_p)\in{\mathbb{E}}^p$.
\end{enumerate}
In each of the four cases the warped product is minimal in ${\mathbb{E}}^{2(h+p)}_{h+p}$.
\end{theorem}

\vskip.1in

\section{\uppercase{Contact $CR$-warped product submanifolds in Sasakian manifolds}}

An odd-dimensional Riemannian manifold $(M,g)$ is called an {\em almost contact metric manifold} if there exist on $M$ a $(1,1)$-tensor field $\phi $, a vector field $\xi $ and a 1-form $\eta $ such that
\begin{align} &\label{16.1}\phi ^{2}X=-X+\eta ( X) \xi , \ \eta (\xi )=1,
\\& \label{16.2} g(\phi X,\phi Y)=g(X,Y) -\eta ( X) \eta ( Y) \end{align}
for  vector fields $X,Y$ on $M$. On an almost contact metric manifold, we also have $\phi \xi =0$ and $\eta \circ \phi =0$. The vector field $\xi$  is called the {\em structure vector field}.

By a {\it contact manifold\/} we mean a  $(2n+1)$-manifold $M$ together with  a global 1-form $\eta$  satisfying $\eta\wedge (d\eta)^n \ne 0$ on $M$.
If $\eta$ of  an almost contact metric manifold $(M,\phi,\xi,\eta,g)$ is a contact form and if  $\eta$ satisfies $d\eta
(X,Y)=g(X,\phi Y)$ for all vectors  $X,Y$ tangent to $M$, then
$M$ is called a {\it contact metric manifold}.
A contact metric manifold is called {\it $K$-contact} if its characteristic vector field $\xi$ is a Killing vector field.
 It is well-known that a contact metric manifold is a $K$-contact manifold if and only if
\begin{align}  \label{16.3}\nabla _{X}\xi =-\phi X \end{align}
holds for all vector fields $X$ on $M$. In fact, an almost contact metric manifold satisfying condition  \eqref{16.3} is also a $K$-contact manifold. Condition \eqref{16.3} is equivalent to \begin{equation}
K(X, \xi)=1 \label{16.4}\end{equation}
for every unit tangent vector $X$  orthogonal to  $\xi$.

An almost contact metric structure of $M$ is called {\em normal} if the Nijenhuis torsion $[\phi ,\phi ]$ of $\phi $ equals to $-2{\rm d}\eta \otimes \xi $. A normal contact
metric manifold is called a {\em Sasakian manifold}. It can be proved that an almost contact metric manifold is Sasakian if and only if the Riemann curvature tensor $R$ satisfies
\begin{equation}\label{16.5} R(X,Y)\xi=\eta(Y)X-\eta(X)Y\end{equation}
 for any vector fields $X,Y$ on $M$.
A Sasakian manifold is also $K$-contact but the converse is not true in general if $\dim M\geq 5$.

A plane section $\pi$ in $T_{p} M$ is called a $\phi$-section if it is spanned by $X$ and $\phi X$, where $X$ is a unit tangent vector orthogonal to $\xi$. The sectional curvature of a $\phi$-section is
called a $\phi$-sectional curvature. A Sasakian manifold with constant $\phi$-sectional curvature
$c$ is said to be a Sasakian space form and is denoted by $\tilde M(c)$.

A submanifold $N$ normal to $\xi$ in a Sasakian manifold is said to be a $C$-totally real
submanifold. In this case, it follows that $\phi$ maps any tangent space of $N$ into the normal space, that is, $\phi(T_{x}N)\subset T^{\perp}_{x}N$, for every $x\in N$.
For submanifolds tangent to the structure vector field $\xi$, there are different classes
of submanifolds. We mention the following.

\begin{enumerate}
\item[(i)] A submanifold $N$ tangent to $\xi$ is called an {\it invariant submanifold} if $\phi$ preserves any tangent space of $M$, that is, $\phi(T_{x}N)\subset T_{x}N$, for every $x\in N$.

\item[(ii)] A submanifold $N$ tangent to $\xi$ is called an {\it anti-invariant submanifold} (or {\it totally real submanifold}) if $\phi$ maps any tangent space of $N$ into the normal space, that is, $\phi(T_{x}N)\subset T^{\perp}_{x}N$, for every $x\in N$.

\item[(iii)] A submanifold $N$ tangent to $\xi$ is called a {\it contact $CR$-submanifold} if it admits
an invariant distribution $\mathcal D$ whose orthogonal complementary distribution ${\mathcal D}^{\perp}$ is anti-invariant, that is, $TN= {\mathcal D} \oplus \mathcal D^{\perp}$ with $\phi(\mathcal D_{x})\subset \mathcal D_{x}$ and $\phi(\mathcal D_{x}^{\perp})\subset T_{x}^{\perp}N$, for every $x\in N$.

\item[(iv)] A contact $CR$-submanifold $N$ of a Sasakian manifold $\tilde M$ is called {\it (contact)
$CR$-product} if it is locally a Riemannian product of a $\phi$-invariant submanifold $N_{T}$ tangent to $\xi$ and a totally real submanifold $N_{\perp}$ of $\tilde M$.\end{enumerate}

First, we present the following two results from \cite{HM03,Mu05}.

\begin{theorem}\label{T:16.1} Let $\tilde M$ be a Sasakian manifold and let  $N=N_{\perp}\times_{f} N_{T}$ be a warped product $CR$-submanifold  such that $N_{\perp}$ is a totally real submanifold and $N_{T}$ an invariant submanifold of $\tilde M$. Then $N$ is a $CR$-product.
\end{theorem}

\begin{theorem}\label{T:16.2} Let  $N=N_{T}\times_{f} N_{\perp}$ be a $CR$-warped products in a Sasakian manifold $\tilde M$. Then we have
\begin{enumerate}

\item The squared norm of the second fundamental form of $N$ satisfies
\begin{align} \label{16.6} ||\sigma||^{2}\geq 2p (||\nabla(\ln f)||^{2}+1),\;\; p= \dim N_{\perp}.\end{align}

\item If the equality sign of \eqref{16.6} holds identically, then $N_{T}$ is a totally geodesic submanifold and $N_{\perp}$ is a totally umbilical submanifold of $\tilde M$. Further, $N$ is a minimal submanifold of $\tilde M$.
Moreover, if $\tilde M$ is ${\bf R}^{2m+1}$ with the usual Sasakian structure then $\ln f$ is an harmonic function.

\item For the case $T^{\perp}N=\phi {\mathcal D}^{\perp}$,  if $p>1$ then the equality sign in \eqref{16.6} holds identically if and only if $N_{\perp}$ is a totally umbilical submanifold of $\tilde M$.
 \end{enumerate}\end{theorem}

For contact $CR$-warped products in Sasakian space forms we have the following from \cite{Mu05} (see also \cite{HM03}).

\begin{proposition} Let  $N=N_{T}\times_{f} N_{\perp}$ be a non-trivial $($i.e., $f$ non constant$\,)$ complete, simply-connected, contact $CR$-warped product those second fundamental form $\sigma$ satisfies 
$$||\sigma||^{2}\geq 2p (||\nabla(\ln f)||^{2}+1)$$ in a Sasakian space form $\tilde M^{2m+1}(c)$. Then we have

\begin{enumerate}
\item $N_{T}$ is a totally geodesic Sasakian submanifold of $\tilde M^{2m+1}(c)$.  Thus $N_{T}$ is a Sasakian space form $N_{T}^{2\alpha+1}(c)$.

\item $N_{\perp}$ is a totally umbilical totally real submanifold of $\tilde M^{2m+1}(c)$. Hence $N_{\perp}$ is a Riemannian manifold of constant sectional curvature. Denote it by $\epsilon$.

\item If $p>1$,  the function $f$ satisfies 
$$||\nabla f||^{2}=\epsilon -\frac{c+3}{4}f^{2}.$$
\end{enumerate}\end{proposition}

The next result from \cite{HM03}  determines the minimum codimension of a contact $CR$-warped product with compact invariant factor in an odd-dimensional sphere endowed with the standard Sasakian structure.

\begin{theorem} \label{T:16.3} Let $N=N_{T}\times_{f} N_{\perp}$ be a contact $CR$-warped product in the
$(2m+1)$-dimensional unit sphere $S^{2m+1}$. If $N_{T}$ is compact, then 
$$m\geq s+p+s p,$$
where $\dim N_{T}=2s+1$ and $\dim N_{\perp}=p$.
\end{theorem}

An easy consequence of Theorem \ref{T:16.3} is the following.

\begin{corollary} Let $\tilde M(c)$ be a Sasakian space form with $c<3$. Then there do not exist contact $CR$-warped products $N_{T}\times_{f} N_{\perp}$, with $N_{T}$ a compact invariant submanifold tangent to $\xi$ and $N_{\perp}$ a $C$-totally real submanifold of $\tilde M(c)$\end{corollary}

The following results from \cite{HM03,Mu05} are the  contact version of Theorems \ref{T:10.1} and \ref{T:10.3}.

\begin{theorem} \label{T:16.4}   Let  $N=N_{T}\times_{f} N_{\perp}$ a contact $CR$-warped product of a Sasakian space form $\tilde M^{2m+1}(c)$ such that $N_{T}$ is a
$(2s+1)$-dimensional invariant submanifold tangent to $\xi$ and $N_{\perp}$ a $p$-dimensional $C$-totally real submanifold of $\tilde M$. Then

\begin{enumerate}
\item The squared norm of the second fundamental form satisfies
\begin{align} \label{16.7} ||\sigma||^{2}\geq 2p \left(||\nabla(\ln f)||^{2}-\Delta(\ln f)+\frac{c+3}{2}s+1\right).\end{align}

\item The equality sign of \eqref{16.7} holds identically if and only if we have:
\item[(2.1)]
  $N_{T}$ is a totally geodesic invariant submanifold of $\tilde M(c)$. Hence $N_{T}$ is a Sasakian space form of constant $\phi$-sectional curvature $c$.
\item[(2.2)] $N_{\perp}$ is a totally umbilical totally real submanifold of $\tilde M$. Hence $N_{\perp}$ is a real space form of sectional curvature $\varepsilon \geq (c+3)/4$.
\end{enumerate} \end{theorem}

An example of contact $CR$-warped submanifold  satisfying the equality case of \eqref{16.7}, but not the equality case of \eqref{16.6} was constructed in \cite{Mu05}.
Also, a contact version of Theorem \ref{3.1} for warped product submanifolds in Sasakian space forms was given in \cite{MM02}.

\vskip.2in

\section{\uppercase{Warped product submanifolds in affine spaces}}

 If $M$ is an $n$-dimensional manifold, let $f:M\to {\bf R}^{n+1}$ be a non-degenerate hypersurface of the affine $(n+1)$-space whose position vector field is nowhere tangent to $M$. Then $f$ can be regarded as a transversal field along itself. We call $\xi=-f$ the centroaffine normal. The $f$ together with this normalization is called a centroaffine hypersurface.

The centroaffine structure equations are given by
\begin{align} \label{23.1} &D_Xf_*(Y)=f_*(\nabla_X Y)+h(X,Y)\xi,\\
\label{23.2} &D_X\xi=-f_*(X),\end{align} where $D$ denotes the canonical flat connection of ${\bf R}^{n+1}$, $\nabla$ is a torsion-free connection on $M$, called the induced centroaffine
connection, and $h$ is a non-degenerate symmetric $(0,2)$-tensor field, called the {\it centroaffine metric}.

  From now on we assume that the centroaffine hypersurface is definite, i.e., $h$ is definite.  In case that $h$ is negative definite, we shall replace $\xi=-f$ by $\xi=f$ for the affine normal. In this way, the second fundamental form $h$ is always positive definite. In both cases, \eqref{23.1}  holds. Equation \eqref{23.2}  change sign. In case $\xi=-f$, we call $M$ positive definite; in case $\xi=f$, we call $M$ negative definite.

Denote by $\hat\nabla$ the Levi-Civita connection of $h$ and by $\hat R$ and $\hat \kappa$ the curvature tensor and the normalized scalar curvature of $h$, respectively. The {\it difference tensor} $K$ is  defined by
\begin{align} \label{23.3} &K_XY=K(X,Y)=\nabla_X Y-\hat\nabla_X Y,\end{align}
which is a symmetric $(1,2)$-tensor field. The difference tensor $K$ and the cubic form $C$ are related by $$C(X,Y,Z)=-2h(K_XY,Z).$$ Thus, for each $X$, $K_X$ is self-adjoint with respect to $h$.
The Tchebychev form $T$ and the {\it Tchebychev vector field} $T^\#$ of $M$ are defined respectively by
\begin{align} \label{23.4} & T(X)=\frac{1}{n} {\rm trace}\,K_X, \\
\label{23.5} &h(T^\#,X)=T(X).\end{align}
If $T=0$ and if we consider $M$ as a hypersurface of the equiaffine space, then $M$ is a so-called  {\it proper affine hypersphere} centered at the origin.

 If the difference tensor $K$ vanishes, then $M$ is a quadric, centered at the origin, in particular an ellipsoid if $M$ is positive definite and a two-sheeted hyperboloid if $M$ is negative definite.

  An affine  hypersurface $\phi: M\to \bf R^{n+1}$ is called a {\it graph hypersurface} if the transversal vector field  $\xi$  is a constant  vector field.   A result of \cite{NP} states that a graph hypersurface $M$ is  locally affine equivalent to the graph immersion of a certain function $F$.  
Again in case that $h$ is nondegenerate, it defines a semi-Riemannian metric,  called the {\it Calabi metric} of the graph hypersurface. If $T=0$,  a graph hypersurface is a so-called  {\it improper affine hypersphere}.

Let $M_1$ and $M_2$ be two improper affine hyperspheres in ${\bf R}^{p+1}$ and ${\bf R}^{q+1}$ defined respectively by the equations: $$x_{p+1}=F_1(x_1,\dots,x_p),\quad y_{q+1}=F_2(y_1,\dots,y_q).$$ Then one can define a new improper affine hypersphere $M$ in ${\bf R}^{p+q+1}$ by
$$z=F_1(x_1,\dots,x_p) + F_2(y_1,\dots,y_q),$$
where $(x_1,\dots,x_p,y_1,\dots,y_q,z)$ are the coordinates on ${\bf R}^{p+q+1}$. The Calabi normal of $M$ is $(0,\dots,0,1)$. Obviously, the Calabi metric on $M$ is the product metric. Following \cite{D2} we call this composition the {\it Calabi
composition} of $M_1$ and $M_2$.

  For a Riemannian $n$-manifold $(M,g)$ with Levi-Civita connection $\nabla$, \'E. Cartan and A. P. Norden studied nondegenerate affine immersions $f:(M,\nabla)\to {\bf R}^{n+1}$ with a transversal vector field $\xi$ and with $\nabla$ as the induced connection. 
  
  The well-known Cartan-Norden theorem states that if $f$ is  such an affine immersion, then  either $\nabla$ is flat and $f$ is a graph immersion or $\nabla$ is not flat and ${\bf R}^{n+1}$ admits a parallel Riemannian metric relative to which $f$ is an isometric immersion and $\xi$ is perpendicular to $f(M)$ (cf. for instance, \cite[p. 159]{N}) (see, also \cite{DNV}).
  
In \cite{c38,c44},  the author studied Riemannian manifolds in affine geometry from a view point different from Cartan-Norden. More precisely, he investigated  the following. 
  \vskip.1in
  
{\bf Realization Problem:}  {\it Which  Riemannian manifolds $(M,g)$ can be immersed  as affine
hypersurfaces in an affine space in such a way that the 
fundamental  form $h$ (e.g. induced by the centroaffine normalization or a constant 
transversal vector field) is the given Riemannian metric $g$}?
\vskip.1in

We say that a Riemannian manifold $(M,g)$ can be {\it realized as an affine hypersurface} if there exists a codimension one affine immersion from $M$ into some affine space in such a way that the induced affine metric $h$ is  exactly the Riemannian  metric $g$ of $M$ (notice that we do not put any assumption on the affine connection).  

Warped product submanifolds in affine hypersurfaces was investigated in \cite{c38,c44}. In particular, it was shown in \cite{c38} that  there exist many warped product Riemannian manifolds  which can be realized either as graph or centroaffine hypersurfaces. More precisely, we have the following results from \cite{c38}.

\begin{theorem}  \label{T:17.1}  Let  $f=f(s)$ be a positive function defined on an open interval $I$. Assume that ${\bf R}, S^n(a^2)$, $H^n(-a^2)$, and $\mathbb E^n$ are equipped with their canonical metrics. Then we have:
\vskip.04in

{\rm (a)} Every warped product surface $I\times_f {\bf R}$ can  be realized as a graph surface in the affine $3$-space ${\bf R}^3$.
\vskip.04in

{\rm (b)}  For each integer $n>2$, the warped product manifold  $I\times_f H^{n-1}(-a^2)$ can  be realized as a graph hypersurface
in ${\bf R}^{n+1}$. 
\vskip.04in

{\rm (c)} If $f'(s)\ne 0$ on $I$, then  the warped product manifold  $I\times_f {\mathbb E}^{n-1},n>2,$  can be realized as a graph hypersurface in ${\bf R}^{n+1}$. 
\vskip.04in

{\rm (d)} If $f'(s)^2>a^2$ on $I$ for some positive number $a$, then the warped product manifold  $I\times_f S^{n-1}(a^2),n>2,$ can  be realized as a graph hypersurface in ${\bf R}^{n+1}$. 
  \end{theorem}
  
  \begin{theorem}  \label{T:17.2} The following results hold.
\vskip.04in

{\rm (a)} If $n>2$ and  $f=f(s)$ is a positive function defined on an open interval $I$, then we have:
\vskip.04in

 {\rm (a.1)}  If $f'(s)^2>f^2(s)-a^2$ on $I$ for some positive number $a$, then $I\times_f H^{n-1}(-a^2)$ can  be realized as a centroaffine hypersurface
in ${\bf R}^{n+1}$. 
\vskip.04in

 {\rm (a.2)} If $f'(s)^2>f(s)^2$ on $I$, then   $I\times_f {\mathbb E}^{n-1}$  can be realized as a centroaffine hypersurface
in ${\bf R}^{n+1}$. 

\vskip.04in 
 {\rm (a.3)} If $f'(s)^2>f(s)^2+a^2$ on $I$ for some positive number $a$, then  $I\times_f S^{n-1}(a^2)$ can  be realized as a graph hypersurface in ${\bf R}^{n+1}$. 
\vskip.04in 

{\rm (b)} If $n=2$ and $f=f(s)$ is a  positive function defined on a closed interval $[\alpha,\beta]$, then the warped product surface $J\times _f {\bf R}, J=(\alpha,\beta)$, can always be realized as a centroaffine surface in ${\bf R}^3$.   
 \end{theorem}

\begin{theorem} \label{T:17.3} If a warped product manifold  $N_1\times_f N_2$  can be realized as a graph hypersurface
in  ${\bf R}^{n+1}$, then the warping function satisfies
\begin{align} \label{17.1}\frac{ \Delta f}{f}\geq -\frac{(n_1+n_2)^2}{4n_2}h(T^{\#},T^\#),\end{align} 
 where $T^{\#}$ is the Tchebychev vector field,  $n=n_1+n_2$,  $n_1=\dim N_1$ and $n_2=\dim N_2$.\end{theorem}

The following result characterizes affine hypersurfaces which verify the equality case of inequality \eqref{17.1}.

\begin{theorem}  \label{T:17.4}  Let  $\phi:N_1\times_f N_2\to {\bf R}^{n+1}$ be a realization of a warped product manifold as  a  graph  hypersurface. If the warping function satisfies the equality case of \eqref{17.1} identically,  then we have: 
\vskip.04in

{\rm (a)} The Tchebychev vector field $T^\#$ vanishes identically.
\vskip.04in

{\rm (b)} The warping function $f$ is a harmonic function.
\vskip.04in

{\rm (c)}  $N_1\times_f N_2$ is realized as an improper affine hypersphere.
  \end{theorem} 

An application of Theorem \ref{T:17.3} is the following.

 \begin{corollary} \label{C:17.1}  If $N_1$ is a compact Riemannian manifold, then every  warped product manifold  $N_1\times_f N_2$ cannot be realized
as an improper affine hypersphere in  ${\bf R}^{n+1}$. \end{corollary}

As an application of Theorems \ref{T:17.3} and \ref{T:17.4} we have the following. 

 \begin{theorem}  \label{T:17.5}  If the Calabi metric of an improper affine hypersphere in an affine space  is the Riemannian product metric of  $k$ Riemannian manifolds,  then the improper affine hypersphere  is locally the Calabi composition of $k$ improper affine spheres.
 \end{theorem} 
 
 Theorem \ref{T:17.3} also implies the following.

\begin{corollary} \label{C:17.2}   If the warping function $f$ of a warped product manifold  $N_1\times_f N_2$ satisfies $\Delta f < 0$  at some point on $N_1$, then $N_1\times_f N_2$ cannot be realized
as an improper affine  hypersphere in  ${\bf R}^{n+1}$. \end{corollary}

 Similarly, for centro-affine hypersurfaces we have the following results from \cite{c38}.

 \begin{theorem}  \label{T:17.6}  If a warped product manifold  $N_1\times_f N_2$ can be realized as a centroaffine hypersurface
in ${\bf R}^{n+1}$, then the warping function satisfies
\begin{align} \label{17.2}\frac{ \Delta f}{f}\geq  n_1\varepsilon- \frac{(n_1+n_2)^2}{4n_2}h(T^{\#},T^\#),\end{align} where $n=n_1+n_2$, $n_i=\dim N_i$, $i=1,2$,  $ \Delta$ is the Laplace operator of $N_1$, and $\varepsilon=1$ or $-1$ according to whether the centroaffine hypersurface is elliptic or hyperbolic. 
\end{theorem}

\begin{theorem}   \label{T:17.7} Let  $\phi:N_1\times_f N_2\to {\bf R}^{n+1}$ be a realization of  a warped product manifold $N_1\times_f N_2$ as  a centroaffine hypersurface. If the warping function satisfies the equality case of \eqref{17.2} identically,  then we have: 
\vskip.04in

{\rm (1)} The Tchebychev vector field $T^\#$ vanishes identically.
\vskip.04in

{\rm (2)} The warping function $f$ is an eigenfunction of the Laplacian $\Delta$ with eigenvalue $n_1\varepsilon$.
\vskip.04in

{\rm (3)}  $N_1\times_f N_2$ is realized as a proper affine hypersphere centered at the origin.  \end{theorem} 

Two immediate consequences of Theorem \ref{T:17.6} are the following.

\begin{corollary} \label{C:17.3}  If the warping function $f$ of a warped product manifold  $N_1\times_f N_2$ satisfies $\Delta f \leq 0$  at some point on $N_1$, then $N_1\times_f N_2$ cannot be realized
as an elliptic  proper affine hypersphere in ${\bf R}^{n+1}$. \end{corollary}
 
\begin{corollary} \label{C:17.4}  If the warping function $f$ of a warped product manifold  $N_1\times_f N_2$ satisfies  $(\Delta f)/ f< -\dim N_1$  at some point on $N_1$, then $N_1\times_f N_2$ cannot be realized
as a hyperbolic  proper affine hypersphere in   ${\bf R}^{n+1}$. \end{corollary}

Some other interesting applications of Theorem \ref{T:17.6} are the following.

 \begin{corollary} \label{C:17.5}  If $N_1$ is a compact Riemannian manifold, then every  warped product manifold  $N_1\times_f N_2$ with arbitrary warping function cannot be realized
as an elliptic  proper affine hypersphere in  ${\bf R}^{n+1}$. \end{corollary}

 \begin{corollary} \label{C:17.6}  If $N_1$ is a compact Riemannian manifold, then every  warped product manifold  $N_1\times_f N_2$ cannot be realized
as an improper affine hypersphere in an affine space ${\bf R}^{n+1}$. \end{corollary}

The following examples  show that the results of this section are optimal.

\begin{example}\label{E:17.1} {\rm Let $M=N_1\times_{\cos s} N_2$  be the warped product of the open interval  $ N_1=(-\pi,\pi)$ and an open portion $N_2$ of the unit $(n-1)$-sphere $S^{n-1}(1)$ equipped with the warped product metric: 
\begin{equation}\begin{aligned}\label{17.3} &h={\rm d}s^2+\cos^2 s\Bigg({\rm d}u_2^2+\cos^2u_2{\rm d}u_3^2+\cdots+\prod_{j=2}^{n-1}\cos^2 u_j{\rm d}u_n^2\Bigg).
\end{aligned}\end{equation}
Consider the immersion of $M$ into the affine $(n+1)$-space  defined by
\begin{equation}\begin{aligned}\label{17.4} &\Bigg(\sin s,\sin u_2\cos s,\ldots,  \sin u_n \cos s\prod_{j=2}^{n-1}\cos^2 u_j ,\cos s\prod_{j=2}^{n}\cos u_j\Bigg). \end{aligned}\end{equation}
Then $M$ is a  centroaffine elliptic   hypersurface whose  centroaffine metric is the warped product metric  \eqref{17.3} and it satisfies $T^\#=0$. Moreover, the warping function $f=\cos s$ satisfies $$\frac{\Delta f}{f}=1=\varepsilon n_1.$$ Hence, this centroaffine hypersurface satisfies the equality case of \eqref{17.2} identically. Consequently, the estimate given in Theorem \ref{T:17.6} is optimal for centroaffine elliptic  hypersurfaces.}
\end{example}

\begin{example}\label{E:17.2} {\rm Let $M={\bf R}\times_{\cosh s} H^{n-1}(-1)$  be  the  warped product of the real line and the unit hyperbolic space $H^{n-1}(-1)$ equipped with warped product metric: 
\begin{equation}\begin{aligned}\label{17.5} &h={\rm d}s^2+\cosh^2 s \Bigg({\rm d}u_2^2+\cosh^2u_2{\rm d}u_3^2+\cdots+\prod_{j=2}^{n-1}\cosh^2 u_j{\rm d}
u_n^2\Bigg).
\end{aligned}\end{equation}
Consider the immersion of $M$ into the affine $(n+1)$-space  defined by
\begin{equation}\begin{aligned}\notag&\Bigg(\sinh s,\sinh u_2\cosh s,\ldots, \sinh u_n \cosh s\prod_{j=2}^{n-1}\cosh^2 u_j ,\cosh s\prod_{j=2}^{n}\cosh u_j\Bigg).\end{aligned}\end{equation}
Then $M$ is a centroaffine hyperbolic  hypersurface whose  centroaffine metric is the warped product metric  \eqref{17.5} and it
satisfies $T^\#=0$. Moreover, the warping function $f=\cosh s$ satisfies $$\frac{\Delta f}{f}=-1=\varepsilon n_1.$$ Therefore, this centroaffine hypersurface satisfies the equality case of \eqref{17.2} identically. Consequently, the estimate given in Theorem \ref{T:17.6} is optimal for centroaffine  hyperbolic hypersurfaces as well.}
\end{example}

\begin{example}\label{E:17.3} {\rm Let $M={\bf R}\times_{s} N_2$  be  the warped product of the real line and an open portion $N_2$ of  $S^{n-1}(1)$ equipped with the warped product metric: 
\begin{equation}\begin{aligned}\label{17.6} &h={\rm d}s^2+ s^2 \Bigg({\rm d}u_2^2+\cos^2u_2{\rm d}u_3^2+\cdots+\prod_{j=2}^{n-1}\cos^2 u_j{\rm d}
u_n^2\Bigg).
\end{aligned}\end{equation}
Consider the immersion of $M$ into the affine $(n+1)$-space  defined by
\begin{equation}\begin{aligned}\label{17.7} & s \Bigg( \sin u_2, \sin u_3\cos u_2,\ldots, \sin u_n \prod_{j=2}^{n-1}\cos^2 u_j , \prod_{j=2}^{n}\cos u_j, \frac{s}{2}\Bigg).
\end{aligned}\end{equation}
Then $M$ is a graph hypersurface with Calabi normal  given by $\xi=(0,\ldots,0,1)$ and it satisfies $T^\#=0$. Moreover, the  Calabi metric of this graph hypersurface is given by the warped product metric  \eqref{17.6}. Clearly, the warping function  is a harmonic function. Therefore, this warped product graph hypersurface satisfies the equality case of \eqref{17.1} identically. Consequently, the estimate given in Theorem \ref{T:17.3} is also optimal.}
\end{example}

\begin{remark} {\rm Example \ref{E:17.1} shows that the conditions $\Delta f\leq 0$ in Corollary \ref{C:17.3}  and  the ``harmonicity''   in Corollary  \ref{C:17.4}  are both necessary. }
\end{remark}

\begin{remark} {\rm Example \ref{E:17.1}  implies that the condition ``$N_1$ is a compact Riemannian manifold'' given   in Corollary \ref{C:17.6}  is  necessary.}
\end{remark}

\begin{remark} {\rm Example \ref{E:17.2} illustrates that the condition $(\Delta f)/f< -\dim N_1$   given in Corollary \ref{C:17.5}  is sharp.}
\end{remark}

\begin{remark} {\rm Example \ref{E:17.3} shows that the condition $\Delta f< 0$ in Corollary \ref{C:17.1}  is optimal as well.}
\end{remark}

\vskip.1in

\section{\uppercase{Twisted product submanifolds}}

Twisted products $B\times_{\lambda} F$ are natural generalizations of warped products, namely the function may depend on both factors (cf. \cite{cbook2}). 
When $\lambda$ depends only on $B$, the twisted product becomes a  warped product.  If $B$ is a point, the twisted product is nothing but a conformal change of metric on $F$.

The study of twisted product submanifolds was initiated in 2000 (see \cite{c5.1}). In particular, the following results are obtained in \cite{c5.1}.

\begin{theorem}\label{T:18.1} If  $\,M=N_\perp\times_\lambda N_T$ is a twisted product $CR$-submanifold of a K\"ahler manifold $\tilde M$ such that $N_\perp$ is a totally real submanifold and $N_T$ is a holomorphic submanifold of $\tilde M$, then $M$ is a $CR$-product. \end{theorem}

\begin{theorem}\label{T:18.2} Let $M=N_T\times_\lambda N_\perp$ be a  twisted product $CR$-submanifold of a K\"ahler manifold $\tilde M$ such that $N_\perp$ is a totally real submanifold
and $N_T$ is a holomorphic submanifold of $\tilde M$. Then we have

\begin{enumerate}
\item The squared norm of the second fundamental form of $M$ in $\tilde M$ satisfies
\begin{equation}\notag ||\sigma||^2\geq 2\, p\, ||\nabla^T(\ln\lambda)||^2,\end{equation}
where $\nabla^T(\ln\lambda)$ is the $N^T$-component of the gradient $\nabla(\ln \lambda)$ of $\,\ln\lambda$ and $p$ is the dimension of $N_\perp$.

\item If $\,||\sigma||^2= 2p\,||\nabla^T\ln\lambda ||^2\,$ holds identically, then $N_T$ is a
totally geodesic submanifold and $N_\perp$ is a totally umbilical  submanifold of $\tilde M$.

\item If $M$ is anti-holomorphic in $\tilde M$ and $\dim N_\perp>1$, then  $\,||\sigma||^2=
2p\,{||\nabla^T\ln\lambda ||^2}\,$ holds identically if and only if $N_T$ is a
totally geodesic  submanifold and $N_\perp$ is a totally umbilical submanifold of $\tilde M$.
\end{enumerate}\end{theorem}

 For mixed foliate twisted product $CR$-submanifolds of K\"ahler manifolds, we have the following result from \cite{c5.1}.

\begin{theorem}\label{18.3} Let $M=N_T\times_\lambda N_\perp$ be a twisted product $CR$-submanifold of a K\"ahler manifold $\tilde M$ such that $N_\perp$ is a totally real submanifold and $N_T$ is a holomorphic submanifold of $\tilde M$. If $M$ is mixed totally geodesic, then we have

\begin{enumerate}
\item The twisted function $\lambda$ is a function on $N_\perp$. 

\item  $N_T\times N^\lambda_\perp$ is a $CR$-product, where $N^\lambda_\perp$ denotes the
manifold $N_\perp$ equipped with the metric $g_{N_\perp}^\lambda=\lambda^2 g_{N_\perp}$.
\end{enumerate}\end{theorem}

Next, we provide ample examples of  twisted product $CR$-submanifolds in complex Euclidean spaces which are not $CR$-warped product submanifolds.

Let $z:N_T\to{\mathbb C}^m$ be a holomorphic submanifold of a complex Euclidean $m$-space ${\mathbb C}^m$ and $w:N^1_\perp\to {\mathbb C}^\ell$ be a totally real submanifold such that the image of $N_T\times N^1_\perp$ under the product immersion $\psi=(z,w)$ does not contain the origin $(0,0)$ of
${\mathbb C}^m\oplus {\mathbb C}^\ell$. 

Let $j:N_\perp^2\to S^{q-1}\subset {\mathbb E}^{q}$ be an isometric immersion of a Riemannian manifold $N_\perp^2$ into the unit hypersphere $S^{q-1}$ of ${\mathbb E}^{q}$ centered at the origin.  

Consider the map
$$\phi=(z, w)\otimes j:N_T\times N^1_\perp\times N^2_\perp\to ({\mathbb C}^m\oplus {\mathbb C}^\ell)\otimes 
{\mathbb E}^{q}$$
 defined by
\begin{equation}\label{Phi}\phi(p_1,p_2,p_3)=(z(p_1),z(p_2)) \otimes j(p_3),\end{equation}
for $ p_1\in N_T, \; p_2\in N^1_\perp,\; p_3\in N^2_\perp$.

On $({\mathbb C}^m\oplus {\mathbb C}^\ell)\otimes {\mathbb E}^{q}$ we define a complex structure $J$ by
$$J((B,E)\otimes F)=({\rm i}B,{\rm i} E)\otimes F,\quad {\rm i}=\sqrt{-1},$$ for any $B\in {\mathbb C}^m,\,E\in {\mathbb C}^\ell$ and $F\in {\mathbb E}^{q}$. Then $({\mathbb C}^m\oplus {\mathbb C} ^\ell)\otimes {\mathbb E}^{q}$ becomes a complex
Euclidean $(m+\ell)q$-space ${\mathbb C}^{(m+\ell)q}$.

Let us put $N_\perp=N^1_\perp\times N^2_\perp$. We denote by $|z|$ the distance function from the origin of ${\mathbb C}^m$ to the position of $N_T$ in ${\mathbb C}^m$ via $z$; and denote by  $|w|$
the distance function from the origin of ${\mathbb C}^\ell$ to the position of $N^1_\perp$ in ${\mathbb C}^\ell$ via $w$. We define a function $\lambda$ by $\lambda=\sqrt{|z|^2+|w|^2}$.
Then $\lambda>0$ is a differentiable function on $N_T\times N_\perp$, which depends on both $N_T$ and $N_\perp=N^1_\perp\times N^2_\perp$. 

Let $M$ denote the twisted product $N_T\times_\lambda N_\perp$ with twisted function $\lambda$. Clearly, $M$  is not a warped product.

For such a twisted product $N_T\times_\lambda N_\perp$ in ${\mathbb C}^{(m+\ell)q}$ defined above we have the following.

\begin{proposition} \label{P:18.1} The map
$$\phi=(z, w)\otimes j:N_T\times_\lambda N_\perp\to {\mathbb C}^{(m+\ell)q} $$ 
defined by \eqref{Phi} satisfies the following properties:

\begin{enumerate}
\item  $\phi=(z, w)\otimes j:N_T\times_\lambda N_\perp\to {\mathbb C}^{(m+\ell)q}$  is an isometric immersion.

\item  $\phi=(z, w)\otimes j:N_T\times_\lambda N_\perp\to
{\mathbb C}^{(m+\ell)q}$ is a twisted product $CR$-submanifold such that $N_T$ is a holomorphic submanifold and $N_\perp$ is a totally real submanifold of ${\mathbb C}^{(m+\ell)q}$.
\end{enumerate}\end{proposition}

Proposition \ref{P:18.1} shows that there are many twisted product $CR$-submanifolds $N_T\times_\lambda N_\perp$ such that $N_T$ are holomorphic submanifolds and $N_\perp$ are totally real submanifolds. Moreover, such twisted product $CR$-submanifolds are not warped product $CR$-submanifolds.

\vskip.2in

\section{\uppercase{Related articles}}

 After the publication of \cite{c5.1,c7,c8,c02-2}, there are more than 100 articles on warped product submanifolds appeared during the last 10 years.  To help further study in this vibrant field of research, we divide those articles  into 16 categories according to their main results as follows.
  \vskip.1in
\begin{enumerate}
\item[1.] {Warped products in Riemannian manifolds:}  \cite{c6.1,c02,c02-2,c04,c05-2,c11,CD08-2,CWei,DT,DV,Ol10-3,suceava}.

\vskip.05in
\item[2.]  {Warped products in (generalized) complex space forms:}  \cite{BRV,c4.0,c4.1,c4.2,c7,c8,c02-2,c02-3,c02-4,c03,c03-2,c04-2,c11,c12,CD08,CD08-2,CM,CV,M04,M05-3,R10,RV}.
 \vskip.05in
  
\item[3.]   {Warped products in K\"ahler manifolds:}  \cite{AK08-2,c7,c02-2,c11,CD08,CD08-2,KAJ08,KKS09,Sa06,Sa09-2,Sa10}.
\vskip.05in

\item[4.]   {Warped products in nearly K\"ahler manifolds:}  \cite{AAK09,KJA,KK09,KKS07,SG08,UC11}.
 \vskip.05in

\item[5.]   {Warped products in locally conformal K\"ahler manifolds:}  \cite{BM04,BM08,JKK10,KJ,KY04,MB08-2,Mu07-2,Y04-2}.
\vskip.05in

\item[6.]   {Warped products in para-K\"ahler manifolds:}  \cite{c11,CMu12,SG08}.
\vskip.05in

 \item[7.]   {Warped products in Sasakian manifolds:} \cite{A,AK08,HM03,Mi04,MM02,Mu05,Mu07,MUKW,Ol08,Ol10,Ol10-2,P,ST08,T,U10-2,UK10,UKK10}.
\vskip.05in

 \item[8.]   {Warped products in generalized Sasakian manifolds:} \cite{AAS06,KC07,KC10-2,KKS07,MN11,MUKW,Ol09-2,SO12,U10-2,UK10}.
\vskip.05in

\item[9.]   {Warped products in Kenmotsu manifolds:} \cite{AK12-2,AEM05,At10,KK08,KKS08,KKS11,MAEM06,Ol09,Sr12,ST10,UOK12}.
 \vskip.05in

 \item[10.]   {Warped products in cosymplectic manifolds:} \cite{At11,HTA12,KC10,Kh12,PKK12,SC07,U10,UA11,UC11.2,UK10,UKK10-2,UWM12,Y04,Y04-2}.
 \vskip.05in

 \item[11.]   {Warped products in other contact metric manifolds:} \cite{A06,AK12,At09-2,At10-2,DK12,KC07,KC10-2,MUKW,P,ST08,SO11,T}.
\vskip.05in

 \item[12.]   {Warped products in Riemannian product manifolds:} \cite{AK12-3,At08,At08-2,At09,At11-2,Sa06-2,Sa09}.
\vskip.05in

 \item[13.]   {Warped products submanifolds in quaternionic manifolds:} \cite{M05,M05-2,V}.
\vskip.05in

\item[14.]  {Warped products in affine space:} \cite{c38,c44,c11}.
 \vskip.05in

 \item[15.]   {Twisted products submanifolds:} \cite{c5.1,KJA,U10-4}.
  \vskip.05in

 \item[16.]   {Warped products submanifolds in other spaces:} \cite{MV12,PKK12,Sa05,SB12,U10-3}.
\end{enumerate}

\end{document}